\documentclass[12pt]{amsart}      
\usepackage{amssymb,enumerate}
\usepackage{amsmath}
\usepackage{setspace}
\usepackage{tcolorbox}

\usepackage{float,enumitem}
\usepackage{mathscinet}

\usepackage{amsmath,amssymb,amsfonts,amsthm}

\newtheorem{theorem}{Theorem}[section] 
\newtheorem{lemma}[theorem]{Lemma} 
\newtheorem{prop}[theorem]{Proposition} 
\newtheorem{corollary}[theorem]{Corollary} 

\theoremstyle{definition} 
\newtheorem{remark}[theorem]{Remark} 
\newtheorem{definition}[theorem]{Definition} 
\newtheorem{example}[theorem]{Example} 

\newcommand{\fin}{\operatorname{Fin}}

\newcommand{\T}{\mathbb{T}}
\newcommand{\U}{\mathcal{U}}
\newcommand{\V}{\mathcal{V}}

\newcommand{\card}{\operatorname{card}}
\newcommand{\pH}{p^{\mathrm{H}}}
\newcommand{\dH}{d^{\mathrm{H}}}

\newcommand\R{\mathbb{R}} 
\newcommand\Z{\mathbb{Z}}
\newcommand\N{\mathbb{N}}

\newcommand{\id}{\mathrm{id}}
\newcommand{\be}{\begin{equation}}
\newcommand{\ee}{\end{equation}}

\newcommand{\hp}{h_{\mathfrak{p}}}
\newcommand{\hm}{h_{\mathfrak{m}}}
\newcommand{\hi}{h_{\mathfrak{i}}}
\newcommand{\tv}{\tau_{\mathrm{V}}}

\newcommand{\Q}{\mathbb Q}

\newcommand{\KK}{\mathcal K}

\newcommand{\abs}[1]{\left\vert#1\right\vert}

\newcommand{\vp}{\varphi}
\newcommand{\B}{{\mathcal{B}}}

\newcommand{\A}{{\mathcal{A}}}

\newcommand{\cl}{\operatorname{cl}}
\newcommand{\ep}{{\varepsilon}}

\newcommand\restr[1]{\raisebox{-.1ex}{$\big\vert$}_{#1}}

\newcommand{\orbn}{{\operatorname{Orb}_n}}
\newcommand\orb[1]{{\operatorname{Orb}_{#1}}}

\newcommand{\coinn}{{\operatorname{Coin}_n}}
\newcommand\coin[1]{{\operatorname{Coin}_{#1}}}

\newcommand{\hsep}{h^{\operatorname{sep}}}
\newcommand{\hspan}{h^{\operatorname{span}}}
\newcommand{\hKT}{h_{\operatorname{KTRT}}}
\newcommand{\hAK}{h_{\operatorname{AKEK}}}

\newcommand{\pnH}{p_n^{\mathrm{H}}}

\usepackage{mathscinet}

\begin{document}

\begin{abstract}
The Adler-Konheim-McAndrew type definitions and the Bowen-Dinaburg-Hood type definitions of parametric topological entropy will be considered on orbits and coincidence orbits of nonautonomous multivalued maps in compact Hausdorff spaces. Their mutual relationship and their link to various further types of definitions like those of (parametric) preimage entropy, will be investigated. In this way, several recent results of the other authors will be generalized and extended into a new setting.
\end{abstract}


\thanks{The authors were supported by the Grant IGA\_PrF\_2024\_006 ``Mathematical Models'' of the Internal Grant Agency of Palack\'y University in Olomouc.}


\title[Parametric topological entropy on orbits]{Parametric topological entropy on orbits of arbitrary multivalued maps in compact Hausdorff spaces}

\author[J. Andres]{Jan Andres}
\address{Department of Mathematical Analysis and Applications of Mathematics, Faculty of Science, Palack\'y University, 17. listopadu 12, 771 46 Olomouc, Czech Republic}
\author[P. Ludv\'\i k]{Pavel Ludv\'\i k}
\email{pavel.ludvik@upol.cz}

\maketitle

\section{Introduction (historical remarks and formulation of problems)}

An alternative title of our paper can read: Topological entropy of multivalued nonautonomous dynamical systems in compact Hausdorff spaces. Topological entropy is one of the most efficient tools for measuring the complexity behaviour of dynamical systems. Its first definition was given in 1965 by Adler, Konheim and McAndrew for single-valued continuous maps in compact topological spaces (see \cite{AKM}). Another definition was introduced in 1971 by Bowen (see \cite{Bo}), and independently by Dinaburg (\cite{Di}), for uniformly continuous maps in not necessarily compact metric spaces. Bowen proved in \cite{Bo} the equivalence of his definition with the one in \cite{AKM} in compact metric spaces. In 1974, Hood extended in \cite{Ho} the Bowen-type definition (via both the notions of separated and spanning subsets) to (topological) uniform spaces.

Later on, Kolyada and Snoha generalized in 1996 the definitions in \cite{AKM} and \cite{Bo,Di} to nonautonomous discrete dynamical systems, given by a sequence of single-valued continuous self-maps, and they have shown that, besides other things, all the definitions become equivalent in compact metric spaces (see \cite{KS}). Very recently, Shao generalized in \cite{Sh} the Bowen-Dinaburg type definitions of Kolyada and Snoha to compact uniform spaces in the sense of Hood and again proved the equivalence of all given definitions under consideration

However, the situation for multivalued maps is much more delicate (see e.g. \cite{AK,A1,A2, AJ,AL1,CMM,CP,EK,KM,KT,RS,RT,SV,WZZ}, for an autonomous case, and e.g. \cite{A3,AL2,AL3}, for a nonautonomous case). Unlike to the mentioned equivalence in the single-valued case, the various definitions of topological entropy for multivalued maps can differ significantly each to other. For instance, those for the mapping $\vp(x)=[0,1]$, for every $x\in[0,1]$, can be equal to $0$ in \cite{A1,AL1,CMM}, but those in \cite{AK,CP,EK,KM,KT,RT} equal to $\infty$. This difference is essentially due to the fact that the definitions deal in the first (zero) case with the image values of given multivalued maps (i.e. sets), while in the second (infinity) case with their (singleton) orbits, resp. orbits of coincidences.

The main aim of the present paper is therefore to generalize jointly the definitions of topological entropy in \cite{AK,EK,KT,RT}, for multivalued autonomous maps in compact metric spaces, and in \cite{KS,Sh}, for single-valued nonautonomous maps in compact uniform spaces, to those on orbits of arbitrary multivalued nonautonomous maps in compact Hausdorff spaces. Such an entropy will be called parametric (whence the title).

Furthermore, apart from some basic properties, we will prove the equivalence of all the generalized related variants (i.e. the Adler-Konheim-McAndrew type and the Bowen-Dinaburg-Hood type), as in the single-valued continuous case.

Because of easier calculations and estimates, we will also consider the parametric topological entropy on coincidence orbits for a special subclass of upper semicontinuous maps, for instance, of admissible maps in the sense of G{\'{o}}rniewicz (see e.g. \cite[Chapter I.G]{AG}), and indicate its relationship to preimage entropies introduced in 1995 by Hurley in \cite{Hu} and further studied e.g. in \cite{CN,HWZ,WZ,YWW,ZZH,FFN,Ni,NP,WZZ}. We will also introduce one definition (see Definition \ref{d:4.2} bellow) having no standard analog in the single-valued case.

The paper is organized as follows. At first, some preliminaries of topological spaces and elements of multivalued analysis will be recalled. Then the Bowen-Dinaburg-Hood type and the Adler-Konheim-McAndrew type definitions of parametric topological entropy on orbits of nonautonomous maps in compact topological spaces will be treated. For multivalued maps determined by selected pairs of single-valued nonautonomous maps, we will also consider parametric topological entropy on coincidence orbits. If one mapping in the selected pairs is identity, then the relationship with various sorts of parametric preimage entropy can be investigated. Several illustrative examples and clarifying comments will be supplied.

\section{Preliminaries}\label{s:2}

The topological spaces under our consideration will be mostly \emph{Hausdorff} (\emph{$T_2$-spaces}), i.e. if $(X,\tau)$ is a topological space with a given topology $\tau$, then for all distinct points $x,y\in X$, there exist disjoint sets $U,V \in\tau$ such that $x\in U$ and $y\in V$.

A collection $\A$ of subsets of a topological space space $(X,\tau)$ is a \emph{cover} of $X$ if their union is all of $X$, i.e. $X = \bigcup\A$. A collection $\A$ is an \emph{open cover} of $X$ if $X = \bigcup\A$ and additionally $\A\subset\tau$. If $\B\subset\A$ is also a cover of $X$, then we say that the cover $\A$ has a \emph{subcover} $\B$ of $X$.

A topological space $(X,\tau)$ is \emph{compact} if every open cover of $X$ has a finite subcover. Every compact Hausdorff space is well known to be uniformizable.

More concretely, every compact Hausdorff space is a complete uniform space with respect to the unique uniformity compatible with the topology. Thus, all the properties of uniform spaces can be employed in compact Hausdorff spaces.

Let us therefore recall the basic properties of uniform spaces with their relation to compact Hausdorff spaces.

\begin{definition}
If $\mathcal{P}$ is a family of pseudometrics on a set $S$ and $p$ is a pseudometric on $S$, then we write $p\ll\mathcal{P}$, provided
\begin{equation*}
\forall\ep>0\, \exists p_\ep\in\mathcal{P} \,\exists\delta>0 \,\forall x,y\in S: p_\ep(x,y)<0 \Rightarrow p(x,y)<\ep.
\end{equation*}

\end{definition}

\begin{definition}\label{d:uniform1}
A \emph{uniform structure} (or a \emph{uniformity}) on a nonempty set $S$ is a family $\U$ of pseudometrics on $S$ with the properties:
\begin{enumerate}
	\item[(U1)] If $p_0,p_1\in\U$, then $\max\{p_0,p_1\}\in\U$.
	\item[(U2)] If $p$ is a pseudometric on $S$ such that $p\ll\U$, then $p\in\U$.
	\item[(U3)] If $x,y\in S$, $x\neq y$, then there is $p\in\U$ such that $p(x,y)>0$.
\end{enumerate}
A \emph{uniform space $X$} is a nonempty set $S$ (set of points of $X$) together with a uniformity on $S$.
\end{definition}

Let $\U = \U(X)$ denote the uniformity of $X$ (i.e., the family of pseudometrics). Denoting, for $p\in\U$, $x\in X$ and $r>0$,
\begin{equation*}
B_p(x,r):= \{y\in X: p(x,y)<r\},
\end{equation*}
the collection of sets $B_p(x,r)$, where $p\in\U$ and $x\in X$, $r>0$, is a base of a Hausdorff topology on the set of points of $X$.

The following generalisation of the Lebesgue covering lemma will be useful in the sequel.

\begin{lemma}[see {\cite[Theorem 33, and the paragraph before]{Ke}}]\label{l:Lebesgue}
Let $(X,\U)$ be a compact uniform space and $\A$ be an open cover of $X$. Then there exist $p\in\U$, $\delta>0$ such that, for every $x\in X$, there exists $A\in\A$ such that 
\begin{equation*}
B_p(x,\delta) \subset A.
\end{equation*}
\end{lemma}

An equivalent (cf. e.g. \cite[p. 188--189]{Ke}) definition of a uniform structure can be given alternatively as follows. Let us note that this approach was employed e.g. by Hood in \cite{Ho} and Shao in \cite{Sh}.

\begin{definition}\label{d:uniform2}
A \emph{uniform structure} $\U$ on a nonempty set $S$ is a collection of subsets of $S\times S$ such that:
\begin{enumerate}
	\item[(i)] If $\gamma\in\U$, then $\Delta_S\subset\gamma$, where $\Delta_S:=\{(s,s): s\in S\}$ stands for the \emph{diagonal} in $S\times S$.
	\item[(ii)] If $\delta\in\U$ and $\delta\subset\gamma\subset S\times S$, then $\gamma\in\U$.
	\item[(iii)] If $\gamma,\delta\in\U$, then $\gamma\cap\delta\in\U$.
	\item[(iv)] If $\gamma\in\U$, then $\gamma^{-1}\in\U$, where $\gamma^{-1}:=\{(s,t):(t,s)\in\U\}$ is the \emph{inverse} of $\gamma\in\U$.
	\item[(v)] If $\gamma\in\U$, then there exists $\delta\in\U$ such that $\delta\circ\delta\subset\gamma$, where $\gamma\circ\delta:=\{(s,v):(s,t)\in\gamma,(t,v)\in\delta, \mbox{ for some $t\in S$}\}$ denotes the \emph{composition} of $\gamma,\delta\in\U$.
	\item[(vi)] \emph{Separation condition:} $\bigcap_{\gamma\in\U} \gamma = \Delta_S$.
\end{enumerate}
\end{definition}

A uniformity $\U$ induces the Hausdorff topology as a family of all subsets $M$ of $X$ such that, there exists a $\gamma\in\U$ satisfying $\gamma[x]\subset M$, where $\gamma[x]:= \{y\in X: (x,y)\in\gamma\}$.

For more details concerning the above types of topological spaces, see e.g. \cite{Is,Ke,Pa}.

It will be also convenient to recall some elementary properties of multivalued maps which will be employed in the sequel. Let $\vp:X\multimap Y$ be a multivalued map and $X,Y$ be topological spaces. All multivalued maps will always have nonempty values, i.e. $\vp:X\to 2^X\setminus\{\emptyset\}$. If the space $Y$ is compact and all the values of $\vp$ are closed, then we can write $\vp:X\to\KK(Y)$, where
\begin{equation*}
\KK(Y):= \{A\subset Y: A\mbox{ is nonempty and compact}\}.
\end{equation*}

The regularity of semicontinuous multivalued maps can be defined by means of the preimages of $\vp\colon X\multimap Y$, where
\begin{equation*}
\vp^{-1}_-(B):= \{x\in X:\vp(x)\subset B\} \mbox{ (``small'' preimage of $\vp$ at $B\subset Y$),}
\end{equation*}
resp.
\begin{equation*}
\vp^{-1}_+(B):= \{x\in X:\vp(x)\cap B\neq\emptyset\} \mbox{ (``large'' preimage of $\vp$ at $B\subset Y$).}
\end{equation*}

\begin{definition}
\begin{enumerate}
	\item[(i)] $\vp:X\multimap Y$ is said to be \emph{upper semicontinuous} (u.s.c.) if $\vp^{-1}_-(B)$ is open for every open $B\subset Y$, resp. $\vp^{-1}_+(B)$ is closed for every closed $B\subset Y$;
	\item[(i)] $\vp:X\multimap Y$ is said to be \emph{lower semicontinuous} (l.s.c.) if $\vp^{-1}_+(B)$ is open for every open $B\subset Y$, resp. $\vp^{-1}_-(B)$ is closed for every closed $B\subset Y$;
	\item[(iii)] $\vp:X\multimap Y$ is said to be \emph{continuous} if it is both u.s.c. and l.s.c.
\end{enumerate}
\end{definition}

Obviously, if $\vp:X\to Y$ is single-valued u.s.c. or l.s.c., then it is continuous. Moreover, if $\vp:X\to\KK(Y)$ is u.s.c. and $K\in\KK(X)$, then $\bigcup_{x\in K} \vp(x)\in\KK(Y)$ (see e.g. \cite[Proposition I.3.20]{AG}, \cite[Corollary 1.2.20]{HP})

\begin{prop}[cf. {\cite[Corollary 1.2.22 and Proposition 1.2.33]{HP}}]\label{p:2.1}
A compact map $\vp:X\to\KK(Y)$ (i.e. $\vp(X)$ is contained in a compact subset of $Y$) is u.s.c. if and only if its graph $\Gamma_\vp:= \{(x,y)\in X\times Y: y\in \vp(x)\}$ is closed.
\end{prop}

Furthermore, if $\vp_1:X\multimap Y$ and $\vp_2:Y\multimap Z$ are u.s.c. (resp. l.s.c., continuous), then so is their composition $\vp_2\circ\vp_1:X\multimap Z$ (see e.g. \cite[Proposition I.2.56]{HP}).

If, in particular, $\vp_1:X\to\KK(Y)$ and $\vp_2:Y\to\KK(Z)$ are u.s.c. (resp. continuous), then so is $\vp_2\circ\vp_1:X\to\KK(Z)$ (see e.g. \cite[Proposition 3.21]{AG}, \cite[Proposition 1.2.56 and Corollary 1.2.20]{HP}).

If $X$ is a Hausdorff topological space and $Y$ is a metric space, then $\vp:X\to\KK(Y)$ is continuous if and only if it is continuous with respect to the Hausdorff metric $\dH$, where
\begin{equation*}
\dH(A,B):= \max\{\sup_{a\in A}(\inf_{b\in B} d(a,b),\sup_{b\in B}(\inf_{a\in A} d(a,b))\},
\end{equation*}
for $A,B\in\KK(Y)$ (see e.g. \cite[Theorem I.3.64]{AG}, \cite[Corollary 1.2.69]{HP}).

\begin{prop}[\cite{Mi51}]\label{p:hyper2}
Let $(X,\U)$ be a uniform space and $(X,\tau)$ be a Hausdorff topological space, where the topology $\tau$ is induced by the uniformity $\U$, as indicated above. Then $(\KK(X),\tv)$, where $\tv$ stands for a Vietoris topology, is a Hausdorff topological space induced by the uniformity $\U^{\mathrm{H}}$, where
\begin{equation*}
\U^{\mathrm{H}}:= \{\pH: p\in \U\mbox{ is bounded}\}
\end{equation*}
and 
\begin{equation*}
\pH(A,B) := \max \{ \sup_{a\in A} (\inf_{b\in B}p(a,b)), \sup_{b\in B} (\inf_{a\in A} p(a,b))\}, \mbox{ $A,B \in \KK(X)$, $p\in\U$.}
\end{equation*}

\end{prop}

It will be also convenient to recall the definition of an important class (from the point of view of applications in the framework of topological fixed point theory) of admissible maps in the sense of G{\'{o}}rniewicz (see e.g. \cite[Chapter I.4]{AG}).

\begin{definition}
Assume that we have a diagram $X\overset{r}{\Leftarrow} Z \overset{q}{\to} Y$, where $r\colon Z\Rightarrow X$ is a (single-valued) Vietoris map and $q\colon Z\to X$ is a (single-valued) continuous map. Then the map $\vp\colon X\multimap Y$ is called \emph{admissible} if it is induced by $\vp(x) = q(r^{-1}(x))$, for every $x\in X$. Thus, we determine the admissible map $\vp$ by the pair $(r,q)$ called an \emph{admissible (selected) pair}.
\end{definition}

Let us recall that a (single-valued) continuous map $r\colon Z\Rightarrow X$ is called the \emph{Vietoris map}, provided
\begin{enumerate}
	\item[(i)] $r$ is onto (i.e. $r(Z) = X$),
	\item[(ii)] $r$ is proper (i.e. $r^{-1}(K)$ is a nonempty compact set, for every compact $K\subset X$),
	\item[(iii)] the given set $r^{-1}(x)$ is \emph{acyclic}, or more precisely \emph{$\Q$-acyclic}, for every $x\in X$ (i.e.
	\begin{equation*}
H_n(r^{-1}(x)) = \begin{cases} 0,\mbox{ for $n>0$,}\\ \Q, \mbox{ for $n=0$,}\end{cases}
\end{equation*}
where $\Q$ stands for the field of rationals and $H_n$ denotes the $n$-dimensional \v{C}ech homology functor with rational coefficients).
\end{enumerate}

Roughly speaking, acyclicity means a homological equivalence to a one point space. When $\Q$ is replaced by the set of integers $\Z$, then we speak about \emph{$\Z$-acyclicity}. An important example of a compact $\Z$-acyclic set is an \emph{$R_\delta$-set}, defined as an intersection of a decreasing sequence of compact contractible sets, i.e. of compact sets homotopically equivalent to a one point space.

For our needs, we will also consider more ``liberal'' classes of multivalued maps, determined by pairs of single-valued maps $(r,q)$, where $r$ need not be necessarily Vietoris.

One can readily check that if $q$ is in particular an identity on $Z$, then $\vp$ reduces to the preimage of $r$, namely $\vp=r^{-1}$, where $X\overset{r}{\leftarrow}Z$.

It allows us to consider for compact $X=Y$ not only the sets of \emph{$n$-orbits} of $\vp$,
\begin{align*}
&\orb{1}(\vp):=X,\\
&\orbn(\vp):=\{(x_0,x_1,\ldots,x_{n-1})\in X^n: x_{i+1}\in\vp(x_i),i=0,1,\ldots,n-2\}, n\geq 2,
\end{align*}
but also the sets of \emph{$n$-orbits of coincidences} of $(r,q)$,
\begin{align*}
&\coin{1}(r,q):= Z,\\
&\coinn(r,q):=\{(z_0,z_1,\ldots,z_{n-1})\in Z^n: q(z_i)=r(z_{i+1}),i=0,1,\ldots,n-2\}, n\geq 2.
\end{align*}

For a sequence of multivalued maps $\vp_{0,\infty}:=\{\vp_j\}_{j=0}^{\infty}$ and for the pairs $(r,q)_{0,\infty}:=\{(r_j,q_j)\}_{j=0}^\infty$ of single-valued maps, i.e.
\begin{equation*}
X\overset{r_j}{\leftarrow}Z\overset{q_j}{\rightarrow} X,
\end{equation*}
where $r_j$ are surjective, for all $j\in\N\cup\{0\}$, we can define the \emph{$n$-orbits} of $\vp_{0,\infty}$ as
\begin{align*}
&\orb{1}(\vp_{0,\infty}):=X,\\
&\orbn(\vp_{0,\infty}):= \{(x_0,x_1,\ldots,x_{n-1})\in X^n: x_{i+1}\in\vp_i(x_i),i=0,1,\ldots,n-2\}, n\geq 2,
\end{align*}
and the \emph{$n$-orbits of coincidences} of $(r,q)_{0,\infty}$ as
\begin{align*}
&\coin{1}((r,q)_{0,\infty}):=Z,\\
&\coinn((r,q)_{0,\infty}):=\{(z_0,z_1,\ldots,z_{n-1})\in Z^n: q_i(z_i)=r_{i+1}(z_{i+1}),i=0,1,\ldots,n-2\},\\
& n\geq 2.
\end{align*}

If $\vp_{0,\infty}=\{\vp_j\}_{j=0}^{\infty}$ is determined by $(r,q)_{0,\infty}=\{(r_j,q_j)\}_{j=0}^{\infty}$ as $\vp_j=q_j\circ r_j^{-1}$, $j\in\N\cup\{0\}$, one can readily check that, because of $x_i=r_i(z_{i})$, $x_{i+1}=q_i(z_i)$ and $x_{i+1}\in\vp_i(x_i)$, $i=0,1,\ldots,n-2$; $n\geq 2$, two different $n$-orbits of coincidences can be determined by a single $n$-orbit, but not vice versa, that two different $n$-orbits cannot be determined by a single $n$-orbit of coincidences. In other words, the notion of orbits of coincidences is obviously more general than the one of orbits.
\begin{lemma}\label{l:1}
Let $X$ be a compact Hausdorff topological space and $\vp_j: X\to\KK(X)$, $j\in\N\cup\{0\}$, be a sequence $\vp_{0,\infty}$ of u.s.c. multivalued maps. Then, for $n\in\N$, the set $\orbn(\vp_{0,\infty})$ of $n$-orbits of $\vp_{0,\infty}$ is compact.
\end{lemma}
\begin{proof}
Since $\orbn(\vp_{0,\infty})\subset X^n$, it suffices to prove that $\orbn(\vp_{0,\infty})$ is closed. We proceed by induction. 

For $n=1$, $\orb{1}(\vp_{0,\infty})=X$, which is closed.

Let $n\in\N$, $n>1$, and assume that $\orb{j}(\vp_{0,\infty})$ is closed for $j<n$. Let us define $\tilde{\vp}_n: \orb{n-1}(\vp_{0,\infty}) \to \KK(X)$ as $\tilde{\vp}_n:= \vp_{n-2} \circ \pi_{n-1}$, where $\pi_{n-1}: X^{n-1} \to X$ is a (continuous) projection to the $(n-1)$-th coordinate. Obviously, $\tilde{\vp}_n$ is a u.s.c. multivalued map, because it is a composition of a u.s.c. map and a continuous map. Thus, $\orbn(\vp_{0,\infty}) = \Gamma_{\tilde{\vp}_n}$, which is closed by Proposition \ref{p:2.1}.
\end{proof}

\begin{lemma}\label{l:1b}
Let $X, Z$ be compact Hausdorff topological spaces and $(r,q)_{0,\infty}$ be a sequence of selected pairs $q_j,r_j:Z\to X$, $j\in\N\cup\{0\}$, where $r_j$, $j\in\N\cup\{0\}$, are single-valued continuous surjections and $q_j$, $j\in\N\cup\{0\}$, are single-valued continuous maps. Then, for $n\in\N$, the set $\coinn((r,q)_{0,\infty})$ of coincidence $n$-orbits of $(r,q)_{0,\infty}$ is compact.
\end{lemma}
\begin{proof}
Since $\coinn((r,q)_{0,\infty})\subset Z^n$, it is sufficient to prove that $\coinn((r,q)_{0,\infty})$ is closed.

For $n=1$, $\coin{1}((r,q)_{0,\infty})=Z$, which is closed.

Let $n\in\N$, $n>1$, and assume that $\coin{j}((r,q)_{0,\infty})$ is closed for $j<n$. Let us define $\tilde{\vp}_n: \coin{n-1}((r,q)_{0,\infty}) \to \KK(X)$ as $\tilde{\vp}_n:= r^{-1}_{n-1} \circ q_{n-2} \circ \pi_{n-1}$, where $\pi_{n-1}: Z^{n-1} \to Z$ is a (continuous) projection to the $(n-1)$-th coordinate. Obviously, $\tilde{\vp}_n$ is a u.s.c. multivalued map with compact values, because it is a composition of a u.s.c. map with compact values and two continuous maps. Thus, $\coinn((r,q)_{0,\infty}) = \Gamma_{\tilde{\vp}_n}$, which is closed by Proposition \ref{p:2.1}.
\end{proof}

\section{Parametric topological entropy on orbits}\label{s:3}
At first, we will generalize the Bowen-Dinaburg type definitions of topological entropy in \cite{KT,RT}, jointly with the Bowen-Dinaburg-Hood type definitions in \cite{Sh}, to those of multivalued nonautonomous arbitrary (i.e. with no regularity restrictions) maps in compact Hausdorff spaces. Some properties of this generalized topological entropy, suitable for its estimation and calculation, will be deduced.

\begin{definition}\label{d:3.1}
Let $X$ be a nonempty set, $p$ a pseudometric on $X$ and $\ep>0$. A set $S\subset X$ is called \emph{$(p,\ep)$-separated} if $p(x,y) > \ep$, for every pair of distinct points $x,y\in S$. A set $R\subset X$ is called \emph{$(p,\ep)$-spanning in $Y\subset X$} if, for every $y\in Y$, there is $x\in R$ such that $p(x,y)\leq \ep$.  
\end{definition}

The following simple observation can be useful.
\begin{lemma}\label{l:sep}
Let $X$, $Y$ be nonempty sets, $p_1$ a pseudometric on $X$, $p_2$ a pseudometric on $Y$ and $\ep>0$. If $\Phi:X\to Y$ is a mapping satisfying
\begin{equation*}
p_1(x,x') = p_2(\Phi(x),\Phi(x')) \mbox{, for every $x,x'\in X$,}
\end{equation*}
then
\begin{enumerate}
	\item[(i)] for every $(p_1,\ep)$-separated set $S\subset X$, there exists a $(p_2,\ep)$-separated set $S'\subset Y$ with the same cardinality,
	\item[(ii)] if, moreover, $\Phi$ is surjective, then for every $(p_2,\ep)$-separated set $S'\subset Y$ there exists a $(p_1,\ep)$-separated set $S\subset X$ with the same cardinality.
\end{enumerate}
\end{lemma}
\begin{proof}

For the proof of (i), consider a $(p_1,\ep)$-separated set $S\subset X$. Since for every $x,x'\in S$
\begin{equation*}
p_2(\Phi(x),\Phi(x')) = p_1(x,x')>\ep,
\end{equation*}
$\Phi(S)\subset Y$ is $(p_2,\ep)$-separated with the same cardinality as $S$.

The statement (ii) is an application of (i) to the mapping $\Phi':Y\to X$ which is an arbitrary single-valued selection of $\Phi^{-1}:Y\multimap X$.
\end{proof}

\begin{definition}\label{d:3.2}
Let $X$ be a nonempty set and $\vp_j:X\multimap X$, $j\in\N\cup\{0\}$, be a sequence $\vp_{0,\infty}$ of multivalued maps. Let $p$ be a pseudometric on $X$, $\ep>0$ and $n\in\N$. Let $p_n$ be a pseudometric on $X^n$ defined as
\begin{align}\label{eq:0}
p_n((x_0,\ldots,x_{n-1}),(y_0,\ldots,y_{n-1})) := \max \{ p(x_i,y_i): i=0,\ldots,n-1\}.
\end{align}
We call $S\subset \orbn(\vp_{0,\infty})$ a \emph{$(p,\ep,n)$-separated set for $\vp_{0,\infty}$} if it is a $(p_n,\ep)$-separated subset. We call $R\subset \orbn(\vp_{0,\infty})$ a \emph{$(p,\ep,n)$-spanning set for $\vp_{0,\infty}$} if it is a $(p_n,\ep)$-spanning subset in $\orbn(\vp_{0,\infty})$.
\end{definition}

Consider a uniform space $(X,\U)$ and a uniform product space $X^n$ whose uniformity $\U^n$ is induced by the set of canonical projections $\pi_i$, $i=0,\ldots,n-1$ (see \cite[Definition 2.26]{Pa}). In particular, these canonical projections are uniformly continuous as mappings between $(X^n,\U^n)$ and $(X,\U)$. Then the pseudometric $p_n$ on $X^n$ (see Definition~\ref{d:3.2}) is a member of $\U^n$. It is an immediate consequence of uniform continuity of $\pi_i$, $i=0,\ldots,n-1$ (see \cite[Definition 1.4]{Pa}) and property (U1) in Definition~\ref{d:uniform1}. In what follows, let us regard $\orbn(\vp_{0,\infty})$ as a subset of $(X^n,\U^n)$.

\begin{lemma}\label{l:3.2}
Let $X$ be a compact topological space, $Y\subset X$, $p$ a continuous pseudometric on $X$ and $\ep>0$. Then
\begin{enumerate}
	\item there exists a finite $(p,\ep)$-spanning set in $Y$,
	\item there exists a finite number $N\in\N$ such that every $(p,\ep)$-separated subset of $X$ has a cardinality not greater than $N$.
\end{enumerate}	
\end{lemma}

\begin{proof}
\begin{enumerate}
\item We firstly assume that $Y$ is a closed set. Then it is also compact and since $Y\subset \bigcup_{x\in Y} B_p(x,\frac{\ep}2)$, there exists a finite subcover $\{B_p(x_k,\frac{\ep}2): k=1,\ldots, N\}$, where $N\in\N$. Consequently, $\{x_k: k=1,\ldots,N\}$ is a finite $(p,\ep)$-spanning set in $Y$.

If $Y$ is an arbitrary set, by the argument above, there exists a $(p,\frac{\ep}2)$-spanning set $R$ in $\overline{Y}$ with a finite cardinality. It is not necessarily a subset of $Y$, but we can construct $R'$ in the following manner. We take one by one all the elements of $R$. Let $x\in R$. If $B_p(x,\frac{\ep}2)\cap A \neq \emptyset$, we include an arbitrary element of this set into $R'$. Clearly, $\card R' \leq \card R$, and $R'$ is a $(p,\ep)$-spanning in $Y$.

\item

Now, applying the step 1, let $R$ be a $(p,\frac{\ep}2)$-spanning set in $Y$ with the cardinality $N\in\N$ and let $S$ be an arbitrary $(p,\ep)$-separated set.

Define $\Phi: S \to R$ by choosing, for each $x\in S$, $\Phi(x)\in R$ with $p(x,\Phi(x))\leq \frac{\ep}2$. Then $\Phi$ is an injective mapping. Otherwise, we obtain a contradiction with the definition of a $(p,\ep)$-separated set. Hence, $\card S \leq \card R = N$.

\end{enumerate}

\end{proof}

Based on the arguments in Lemma~\ref{l:3.2} and Definition~\ref{d:3.2}, we can proceed in giving the correct definitions of topological entropy for arbitrary multivalued maps as follows.

\begin{definition}\label{d:hsepspan}
Let $X=(X,\U)$ be a compact uniform space and $\vp_j: X\multimap X$, $j\in\N\cup\{0\}$, be a sequence $\vp_{0,\infty}$ of multivalued  maps. Denoting by $s(\vp_{0,\infty},p,\ep,n)$ the largest cardinality of a $(p,\ep,n)$-separated set for $\vp_{0,\infty}$, we take
\begin{equation*}
\hsep(\vp_{0,\infty},p,\ep) := \limsup_{n\to\infty} \frac1{n}\log s(\vp_{0,\infty},p,\ep,n).
\end{equation*}
The \emph{topological entropy} $\hsep(\vp_{0,\infty})$ of $\vp_{0,\infty}$ is defined to be
\begin{equation*}
\hsep(\vp_{0,\infty}):= \sup_{p\in\U,\ep>0} \hsep(\vp_{0,\infty},p,\ep).
\end{equation*}

Denoting by $r(\vp_{0,\infty},p,\ep,n)$ the smallest cardinality of a $(p,\ep,n)$-spanning set for $\vp_{0,\infty}$, we take
\begin{equation*}
\hspan(\vp_{0,\infty},p,\ep) := \limsup_{n\to\infty} \frac1{n}\log r(\vp_{0,\infty},p,\ep,n).
\end{equation*}
The \emph{topological entropy} $\hspan(\vp_{0,\infty})$ of $\vp_{0,\infty}$ is defined to be
\begin{equation*}
\hspan(\vp_{0,\infty}):= \sup_{p\in\U,\ep>0} \hspan(\vp_{0,\infty},p,\ep).
\end{equation*}
\end{definition}

\begin{remark}
The existence of the largest cardinality of a $(p,\ep,n)$-separated set for $\vp_{0,\infty}$ is, in view of Definition~\ref{d:3.2}, equivalent to the existence of the largest cardinality of a $(p_n,\ep)$-separated subset of $\orbn(\vp_{0,\infty})$. It is guaranteed by Lemma~\ref{l:3.2}, because $(X^n,\U^n)$ is a compact uniform space and $p_n\in \U^n$. Moreover, both cardinalities $s(\vp_{0,\infty},p,\ep,n)$, $r(\vp_{0,\infty},p,\ep,n)$ are finite by the same arguments.
\end{remark}

\begin{remark}
The correctness of Definition~\ref{d:hsepspan} can be also proved in an alternative way, because each arbitrary map $\vp_j$, $j\in\N\cup\{0\}$, from the sequence $\vp_{0,\infty}$ is a multivalued selection $\vp_j\subset \tilde{\vp}_j$ of a u.s.c. mapping $\tilde{\vp}_j$, $j\in\N\cup\{0\}$, whose graph $\Gamma_{\tilde{\vp}}$ is the closure of the graph $\Gamma_\vp$ of $\vp_j$, i.e. $\Gamma_{\tilde{\vp}_j} = \cl_{X\times X} \Gamma_{\vp_j}$, $j\in\N\cup\{0\}$. For the sequence $\tilde{\vp}_{0,\infty} = \{\tilde{\vp}_j\}_{j=0}^{\infty}$ of u.s.c. maps $\tilde{\vp}_j$, $j\in\N\cup\{0\}$, the existence of maxima (for separated sets), resp. minima (for spanning sets) can be proved quite analogously as in the particular autonomous case in a compact metric space (see \cite{KT}, and cf. Lemma~\ref{l:1}).
\end{remark}

\begin{lemma}\label{l:sep-span}
Let $X$ be a nonempty set, $p$ a pseudometric on $X$, $Y\subset X$, $\ep>0$ and $S\subset Y$ be a maximal (with respect to inclusion) $(p,\ep)$-separated set.
\begin{enumerate}
	\item[(i)] Then $S$ is a $(p,\ep)$-spanning set in $Y$.
	\item[(ii)] If, moreover, $R$ is a $(p,\frac{\ep}2)$-spanning set in $Y$, then $\card R \geq \card S$.
\end{enumerate}
\end{lemma}
\begin{proof}
\begin{enumerate}
	\item[(i)] Let $S\subset Y$ be a $(p,\ep)$-separated set with the maximal cardinality. Then, for every $y\in Y$, there exists $x\in S$ such that $p(x,y)\leq\ep$. Hence, $S$ is a $(p,\ep)$-spanning set in $Y$.
	\item[(ii)] Let $S\subset Y$ be a $(p,\ep)$-separated set with the maximal cardinality and $R$ be any $(p,\frac{\ep}2)$-spanning set in $Y$. Then we can define the mapping $\Phi: S \to R$ such that, for every $x\in S$, we take $\Phi(x)\in R$ with $p(x,\Phi(x))\leq\frac{\ep}2$. Since this mapping is injective, $\card R \geq \card S$.
\end{enumerate}
\end{proof}

\begin{remark}\label{r:zorn}
If $X$ is a nonempty set, $p$ a pseudometric on $X$, $Y\subset X$, and $\ep>0$, then there always exists at least one maximal $(p,\ep)$-separated subset of $Y$. Indeed, it suffices to consider the set of $(p,\ep)$-separated subsets of $Y$ partially ordered by the inclusion and apply Zorn's lemma (every chain of $(p,\ep)$-separated subset of $Y$ has an upper bound being a union of all the elements of the chain).
\end{remark}

\begin{theorem}\label{t:KT-sep-span}
Let $X=(X,\U)$ be a compact uniform space and $\vp_j: X\multimap X$, $j\in\N\cup\{0\}$, be a sequence $\vp_{0,\infty}$ of multivalued  maps. Then $\hsep(\vp_{0,\infty})=\hspan(\vp_{0,\infty})$.
\end{theorem}
\begin{proof}
Taking into account Definition~\ref{d:hsepspan}, condition (i) in Lemma~\ref{l:sep-span} gives $\hspan(\vp_{0,\infty}) \leq \hsep(\vp_{0,\infty})$. The inequality $\hsep(\vp_{0,\infty}) \leq \hspan(\vp_{0,\infty})$ follows from condition (ii) of the same lemma.
\end{proof}

\begin{definition}\label{d:hKT}
Let $X=(X,\U)$ be a compact uniform space and $\vp_j: X\multimap X$, $j\in\N\cup\{0\}$, be a sequence $\vp_{0,\infty}$ of multivalued  maps. Then (along the lines of the papers \cite{KT,RT}, whose authors have the initials KT and RT)
\begin{equation*}
\hKT(\vp_{0,\infty}):= \hsep(\vp_{0,\infty})=\hspan(\vp_{0,\infty}).
\end{equation*}
\end{definition}

\begin{remark}
Definition~\ref{d:hKT} generalizes its analog in \cite{KT,RT} for multivalued autonomous u.s.c. maps in compact metric spaces (for continuous dynamical systems, cf. \cite{RS}) as well as those in \cite{Sh} for single-valued nonautonomous continuous maps in compact uniform spaces. Moreover, for single-valued autonomous continuous maps in compact metric spaces, resp. compact uniform spaces, it coincides with the standard definitions in \cite{Bo,Di,Ho}. The same is true with respect to the Bowen-Dinaburg type definitions for single-valued nonautonomous continuous maps in compact metric spaces (i.e. for nonautonomous dynamical systems) in \cite{KS}.
\end{remark}

\begin{definition}
Let $X=(X,\U)$ be a compact uniform space, $\vp_j: X\multimap X$, $j\in\N\cup\{0\}$, be a sequence $\vp_{0,\infty}$ of multivalued  maps and 
$\psi_j: Y\multimap Y$, $j\in\N\cup\{0\}$, be a sequence $\psi_{0,\infty}$ of multivalued  maps. We say that $\psi_{0,\infty}$ is \emph{equi-semi-conjugate} to $\vp_{0,\infty}$ if there exists a sequence of equi-continuous surjections $f_j:X\to Y$, $j\in\N\cup\{0\}$ such that for every $x\in X$ and $j\in\N\cup\{0\}$:
\begin{equation*}
\psi_j \circ f_j(x) \subset f_{j+1} \circ \vp_j(x).
\end{equation*} 
\end{definition}

\begin{theorem}\label{t:semi}
Let $X=(X,\U)$ be a compact uniform space, $\vp_j: X\multimap X$, $j\in\N\cup\{0\}$, be a sequence $\vp_{0,\infty}$ of multivalued  maps and 
$\psi_j: Y\multimap Y$, $j\in\N\cup\{0\}$, be a sequence $\psi_{0,\infty}$ of multivalued  maps. If $\psi_{0,\infty}$ is equi-semi-conjugate to $\vp_{0,\infty}$, then
\begin{equation*}
\hKT(\psi_{0,\infty}) \leq \hKT(\vp_{0,\infty}).
\end{equation*}
\end{theorem}

\begin{proof}
Let $\ep>0$ and $p\in\U$. Since $X$ is a compact uniform space, $f_j:X\to Y$, $j\in\N\cup\{0\}$, is a sequence of uniformly equi-continuous maps and $p$ is uniformly continuous. Then there exists $\delta>0$ such that, for every $j\in\N\cup\{0\}$ and every $a,b\in X$, $p(a,b) < \delta$ implies $p(f_j(a),f_j(b))<\frac{\ep}2$.

Given $n\in\N$, we define $\Phi_n: \orbn(\vp_{0,\infty})\to Y^n$ as
\begin{equation*}
\Phi_n(x_0,\ldots,x_{n-1}) = (f_0(x_0),\ldots,f_{n-1}(x_{n-1})).
\end{equation*}

We claim that $\orbn(\psi_{0,\infty}) \subset \Phi_n(\orbn(\vp_{0,\infty}))$. Let $(y_0,\ldots,y_{n-1})\in\orbn(\psi_{0,\infty})$. We proceed by an inductive construction of $(x_0,\ldots,x_{n-1})\in\orbn(\vp_{0,\infty})$ such that $\Phi_n(x_0,\ldots,x_{n-1}) = (y_0,\ldots,y_{n-1})$. 

By surjectivity, there exists $x_0\in f_0^{-1}(y_0)$. Let $m\in\N$, $1<m<n$, be such that $(x_0,\ldots, x_{m-2})\in \orb{m-1}(\vp_{0,\infty})$ and $x_j \in f_j^{-1}(y_j)$, for $j=0,\ldots,m-2$. Then 
\begin{equation*}
y_{m-1} \in \psi_{m-2}(y_{m-2}) = \psi_{m-2}\circ f_{m-2} (x_{m-2}) \subset f_{m-1}\circ \vp_{m-2}(x_{m-2}).
\end{equation*}
There exists $x_{m-1}\in \vp_{m-2}(x_{m-2})$ such that $f_{m-1}(x_{m-1})=y_{m-1}$. Thus, $(x_0,\ldots,x_{m-1})\in \orb{m}(\vp_{0,\infty})$ and $\Phi_m(x_0,\ldots,x_{m-1}) = (y_0,\ldots,y_{m-1})$ which completes the inductive argument.


For a given $n\in\N$ and $p\in\U$, let $R$ be an $(p,\delta,n)$-spanning set for $\vp_{0,\infty}$ with the minimal cardinality. We set $T = \Phi_n(R)$. Then, for every $\bar{y}\in\orbn(\psi_{0,\infty})$, there exists $\bar{x}\in\orbn(\vp_{0,\infty})$ such that $\Phi_n(\bar{x})=\bar{y}$. Since $R$ is $(p,\delta,n)$-spanning for $\vp_{0,\infty}$, there exists $\bar{s}:=(s_0,\ldots,s_{n-1})\in R$ such that $p_n(\bar{s},\bar{x})<\delta$. Thus, by the choice of $\delta$, also $p_n((f_0(s_0),\ldots,f_{n-1}(s_{n-1})),\bar{y})<\frac{\ep}2$. 

If $T$ is a subset of $\orbn(\psi_{0,\infty})$, it is a $(p,\frac{\ep}2,n)$-spanning set for $\psi_{0,\infty}$. Otherwise, we modify $T$ similarly as in the proof of Lemma~\ref{l:3.2} to get $T'$, $\card{T'}\leq\card{T}$, which is a $(p,\ep,n)$-spanning set for $\psi_{0,\infty}$.

Therefore, for all $n\in\N$,
\begin{equation*}
r(\vp_{0,\infty},p,\delta,n) = \card(R) \geq \card(T') \geq r(\psi_{0,\infty},p,\ep,n).
\end{equation*}
Hence,
\begin{equation*}
\hspan(\vp_{0,\infty}) \geq \hspan(\vp_{0,\infty},p,\delta) \geq \hspan(\psi_{0,\infty},p,\ep),
\end{equation*}
and finally $\hspan(\vp_{0,\infty}) \geq \hspan(\psi_{0,\infty})$.

\end{proof}

\begin{remark}\label{r:3.4}
In an important particular case of Theorem~\ref{t:semi}, when $\psi_j\subset\vp_j$, $j\in\N\cup\{0\}$, the entropy inequality $\hKT(\psi_{0,\infty})\leq \hKT(\vp_{0,\infty})$ follows directly from Theorem~\ref{t:semi}, when $f_j=\id\restr{X}$, $X=Y$, $j\in\N\cup\{0\}$. Observe that this entropy inequality is, under an even more general assumption $\orbn(\vp_{0,\infty}) \subset \orbn(\psi_{0,\infty})$, $n\in\N$, an immediate consequence of Definition~\ref{d:hKT}.
\end{remark}

\begin{definition}
Let $(X,\U)$ be a compact uniform space, $\vp_j: X\multimap X$, $j\in\N\cup\{0\}$, be a sequence $\vp_{0,\infty}$ of multivalued  maps, and let $k\in\N$. We define the \emph{$k$-th iterate} $\vp^{[k]}_{0,\infty}$ of $\vp_{0,\infty}$ as $\{\vp^{[k]}_{jk}\}_{j=0}^{\infty}$, where $\vp^{[k]}_{jk} = \vp_{jk+k-1}\circ \ldots \circ \vp_{jk}$, for each $j\geq 0$, $k\geq 1$.
\end{definition}

The following lemma is a formal generalization of \cite[Lemma 5.3]{KT}, although it deals with non-autonomous dynamical systems and, moreover, it drops a regularity assumption imposed on upper semi-continuous mappings. Nevertheless, the proof is exactly the same, and so we omit it here. The only difference is conceptual; we use a generalized definition of the system of $n$-orbits, namely $\orbn(\vp_{0,\infty})$.

\begin{lemma}\label{l:KT}
Let $(X,\U)$ be a compact uniform space, $p\in\U$, $\vp_j: X\multimap X$, $j\in\N\cup\{0\}$, be a sequence $\vp_{0,\infty}$ of multivalued  maps, and let $k\in\N$. Then, for all $n\in\N$ and $\ep>0$, if $m\in\N$ is chosen such that $(m-1)k < n \leq mk$, then
\begin{equation*}
s(\vp_{0,\infty},p,\ep,n) \leq \left( s(\vp^{[k]}_{0,\infty},p,\frac{\ep}2,m)\right)^k.
\end{equation*}
\end{lemma}

Also the proof of the following theorem is similar to the one in \cite[Theorem 5.4]{KT}, but since it is an important result, we supply the proof to expose all the differences.

\begin{theorem}\label{t:3.2}
Let $(X,\U)$ be a compact uniform space, $\vp_j: X\multimap X$, $j\in\N\cup\{0\}$, be a sequence $\vp_{0,\infty}$ of multivalued  maps, and let $k\in\N$. Then
\begin{equation*}
\hKT(\vp_{0,\infty}) \leq \hKT(\vp^{[k]}_{0,\infty}) \leq k \cdot \hKT(\vp_{0,\infty}).
\end{equation*}
\end{theorem}
\begin{proof}
We show that $\hKT(\vp^{[k]}_{0,\infty})\leq k \cdot \hKT(\vp_{0,\infty})$. Let $S$ be a $(p,\ep,n)$-separated set for $\vp^{[k]}_{0,\infty}$ with the maximal cardinality. Then $S$ is a $(p_n,\ep)$-separated subset of $\orbn(\vp^{[k]}_{0,\infty})$. For every $(x_0,\ldots,x_{n-1})\in S$, we can find $(y_0,\ldots,y_{nk-1})\in\orb{nk}(\vp_{0,\infty})$ such that $y_{ik}=x_i$, for $i=1,\ldots,n-1$. Let us denote the set of these elements by $\tilde{S}$. Then $\card{\tilde{S}} = \card{S}$, and $\tilde{S}$ is a $(p,\ep,nk)$-separated set for $\vp_{0,\infty}$. Hence,
\begin{align*}
s(\vp^{[k]}_{0,\infty},p,\ep,n) &\leq s(\vp_{0,\infty},p,\ep,nk),\\
\limsup_{n\to\infty} \frac1{n}\log s(\vp^{[k]}_{0,\infty},p,\ep,n) &\leq k \limsup_{n\to\infty} \frac1{nk} \log s(\vp_{0,\infty},p,\ep,nk)\\
&\leq k \limsup_{m\to\infty} \frac1{m} \log s(\vp_{0,\infty},p,\ep,m).
\end{align*}
After passing to the supremum over $p\in\U$ and $\ep>0$, we obtain the inequality $\hKT(\vp^{[k]}_{0,\infty})\leq k \cdot \hKT(\vp_{0,\infty})$.

To prove the inequality $\hKT(\vp_{0,\infty}) \leq \hKT(\vp^{[k]}_{0,\infty})$, we apply Lemma~\ref{l:KT}. Given $n\in\N$, we consider the unique $m_n\in\N$ such that $(m_n-1)k<n\leq m_nk$. Then
\begin{equation*}
s(\vp_{0,\infty},p,\ep,n) \leq \left( s(\vp^{[k]}_{0,\infty},p,\frac{\ep}2,m_n)\right)^k.
\end{equation*}
Inferring that $\frac{k}{n}<\frac1{m_n-1}$, for $n>1$, and $\lim_{n\to\infty} m_n = \infty$, we get
\begin{align*}
&\limsup_{n\to\infty} \frac1{n}\log s(\vp_{0,\infty},p,\ep,n) \leq \limsup_{n\to\infty} \frac1{n} \log \left( s(\vp^{[k]}_{0,\infty},p,\frac{\ep}2,m_n)\right)^k\\
&= \limsup_{n\to\infty} \frac{k}{n} \log s(\vp^{[k]}_{0,\infty},p,\frac{\ep}2,m_n) \leq \limsup_{m\to\infty} \frac1{m_n-1} \log s(\vp^{[k]}_{0,\infty},p,\frac{\ep}2,m_n)\\
&\leq \limsup_{m\to\infty} \frac1{m-1} \log s(\vp^{[k]}_{0,\infty},p,\frac{\ep}2,m) = \limsup_{m\to\infty} \frac1{m} \log s(\vp^{[k]}_{0,\infty},p,\frac{\ep}2,m).
\end{align*}
Taking a supremum over $p\in\U$ and $\ep>0$, we conclude that $\hKT(\vp_{0,\infty}) \leq \hKT(\vp^{[k]}_{0,\infty})$.

\end{proof}

Now, the Adler-Konheim-McAndrew type definitions will be under consideration.

\begin{definition}\label{d:3.7}
Let  $X$ be a compact topological space and $\vp_j: X\multimap X$, $j\in\N\cup\{0\}$, be a sequence $\vp_{0,\infty}$ of multivalued  maps. For an open cover $\A$ of $X$ and $n\in\N$, $n\geq 2$, we define
\begin{equation*}
\A^n:= \{ U_0\times\ldots\times U_{n-1}: U_0,\ldots,U_{n-1}\in\A\}.
\end{equation*}
Letting $N(\vp_{0,\infty},\A,n)$ be a minimal cardinality of a subcover of $\A^n$ covering $\orbn(\vp_{0,\infty})$, let us put (along the lines of the papers \cite{AK,EK}, whose authors have the initials AK and EK)
\begin{align*}
\hAK(\vp_{0,\infty}) := \sup_{\A} \limsup_{n\to\infty} \frac1{n} \log N(\vp_{0,\infty},\A,n),
\end{align*}
where the supremum is taken over all open covers $\A$ of $X$.
\end{definition}

\begin{remark}
Definition~\ref{d:3.7} generalizes its particular cases for multivalued autonomous u.s.c. maps $\vp:X\to\KK(Y)$, $\vp=\vp_j$, $j\in\N\cup\{0\}$, in a compact metric space $X$ in \cite[Definition 2.2]{AK}, and in a unit cube $[0,1]^n$ in \cite{EK}. Its slightly more general analog for a sequence $f_{0,\infty}$ of single-valued continuous maps $f_j:X\to X$, $j\in\N\cup\{0\}$, was given in \cite[Section 2]{KS}.
\end{remark}

\begin{theorem}\label{t:KT=AK}
Let $X$ be a compact Hausdorff space and $\vp_j: X\multimap X$, $j\in\N\cup\{0\}$, be a sequence $\vp_{0,\infty}$ of multivalued  maps. Then $\hKT(\vp_{0,\infty})=\hAK(\vp_{0,\infty})$.
\end{theorem}

\begin{proof}
Every compact Hausdorff space is (uniquely) uniformizable. Therefore, there is a unique (compact) uniform space $(X,\U)$ with a compatible topology.

Let $\ep>0$ and $p\in\U$. We take an open cover $\A:=\{B_p(x,\frac{\ep}2): x\in X\}$ of $X$. Let $S$ be a $(p_n,\ep)$-separated subset of $\orbn(\vp_{0,\infty})$ with the minimal cardinality. Then, for two distinct elements $(x_0,\ldots,x_{n-1})$, $(y_0,\ldots,y_{n-1})$ of $S$, there exists $j\in\{0,\ldots,n-1\}$ such that $p(x_j,y_j)\geq\ep$. Consequently, $x_j,y_j$ cannot be together elements of one $A\in\A$. Hence,
\begin{align*}
s(\vp_{0,\infty},p,\ep,n) = \card S &\leq N(\vp_{0,\infty},\A,n),\\
\limsup_{n\to\infty} \frac1{n}\log s(\vp_{0,\infty},p,\ep,n) &\leq \limsup_{n\to\infty} \frac1{n}\log N(\vp_{0,\infty},\A,n)\leq \hAK(\vp_{0,\infty}),\\
\hsep(\vp_{0,\infty}) &\leq \hAK(\vp_{0,\infty}).
\end{align*}

Now, let $\A$ be an open cover of $X$. We apply Lemma~\ref{l:Lebesgue} on $\A$ which provides $p\in\U$, $\delta>0$ such that, for every $x\in X$, there is $A\in\A$ such that $B_p(x,\delta)\subset A$.

For $n\in\N$, let $R$ be a $(p_n,\frac{\delta}3)$-spanning set in $\orbn(\vp_{0,\infty})$ with the minimal cardinality. Then $\B:=\{B_{p_n}(\bar{x},\frac{\delta}2): \bar{x}\in R\}$ is clearly an open cover of $\orbn(\vp_{0,\infty})$. Moreover, for every $\bar{x}\in S$, there exists $A\in\A^n$ such that $B_{p_n}(\bar{x},\frac{\delta}2)\subset A$. Thus,
\begin{align*}
 N(\vp_{0,\infty},\A,n) &\leq \card R = r(\vp_{0,\infty},p,\frac{\delta}2,n),\\
\limsup_{n\to\infty} \frac1{n}\log N(\vp_{0,\infty},\A,n) &\leq \limsup_{n\to\infty} \frac1{n}\log r(\vp_{0,\infty},p,\frac{\delta}2,n) \leq \hspan(\vp_{0,\infty}),\\
\hAK(\vp_{0,\infty}) &\leq \hspan(\vp_{0,\infty}).
\end{align*}

Applying Definition~\ref{d:hKT}, we get the desired conclusion, i.e. $\hKT(\vp_{0,\infty})=\hAK(\vp_{0,\infty})$.

\end{proof}

\begin{remark}
For multivalued autonomous u.s.c. maps $\vp:X\to\KK(X)$, $\vp=\vp_j$, $j\in\N\cup\{0\}$, in a compact metric space $X$, it was just pointed out in \cite{AK} that Theorem~\ref{t:KT=AK} is a consequence of \cite[Theorem 2.2]{RT}, \cite[Theorem 3.1]{KT}, and \cite[Theorem 4.3]{EK}. Let us note that the argument in \cite[Theorem 4.3]{EK} is restricted to a unit interval $[0,1]$, resp. later to a unit cube $[0,1]^n$. For a sequence $f_{0,\infty}$ of single-valued continuous maps $f_j:X\to X$, $j\in\N\cup\{0\}$, Theorem~\ref{t:KT=AK} was proved in a slightly more general form in \cite[Theorem 2.4]{Sh}.
\end{remark}

\begin{definition}\label{d:3.8}
Let $X$ be a compact Hausdorff space and $\vp_j:X\multimap X$, $j\in\N\cup\{0\}$, be a sequence $\vp_{0,\infty}$ of multivalued maps. Then, in view of Theorem~\ref{t:KT=AK}, we can put
\begin{equation*}
h(\vp_{0,\infty}) := \hKT(\vp_{0,\infty})=\hAK(\vp_{0,\infty}).
\end{equation*}
\end{definition}

\section{Parametric topological entropy on orbits of coincidences}\label{s:4}

Let $(r,q)_{0,\infty}$ be the sequence of selected pairs $(r_j,q_j)_{j=0}^{\infty}$, where $X=(X,\U)$, $Z=(Z,\V)$ are compact uniform spaces and $X\overset{r_j}{\leftarrow}Z\overset{q_j}{\rightarrow} X$, $j\in\N\cup\{0\}$, are single-valued continuous maps such that $r_j$, $j\in\N\cup\{0\}$, are surjections.

Let $p\in\U$, $n$ be a positive integer and $\ep>0$. We say that the subset $S\subset\coinn((r,q)_{0,\infty})$ is \emph{$(p,\ep,n)$-separated for $(r,q)_{0,\infty}$}, resp. the subset $R\subset\coinn((r,q)_{0,\infty})$ is \emph{$(p,\ep,n)$-spanning for $(r,q)_{0,\infty}$}, if it is a $(p^*_n,\ep)$-separated subset of $\coinn((r,q)_{0,\infty})$, resp. a $(p^*_n,\ep)$-spanning set in $\coinn((r,q)_{0,\infty})$, where
\begin{align*}
p^*_n((z_0,\ldots,z_{n-1})&,(z'_0,\ldots,z'_{n-1}))\\
 &:=\max_{i=0,\ldots,n-1} \max\{p((r_i(z_i)),r_i(z'_i)), p(q_i(z_i),q_i(z'_i))\},
\end{align*}
for $(z_0,\ldots,z_{n-1}),(z'_0,\ldots,z'_{n-1})\in Z^n$.


%

Adopting Lemma~\ref{l:sep-span}, we can proceed in accordance with Definition~\ref{d:hKT} by the following definition.

\begin{definition}\label{d:4.1}
We call the \emph{parametric topological entropy} (on orbits of coincidences) $h((r,q)_{0,\infty})$, for a family $(r,q)_{0,\infty}$ of selected pairs $(r_j,q_j)^{\infty}_{j=0}$, provided
\begin{align*}
h((r,q)_{0,\infty})&:= \sup_{\ep>0,p\in\U} \limsup_{n\to\infty}\frac1{n}\log s((r,q)_{0,\infty},p,\ep,n) \\
&= \sup_{\ep>0,p\in\U} \limsup_{n\to\infty}\frac1{n}\log r((r,q)_{0,\infty},p,\ep,n),
\end{align*}
where $s((r,q)_{0,\infty},p,\ep,n)$ stands for the maximal cardinality of a $(p,\ep,n)$-separated set for $(r,q)_{0,\infty}$, and $r((r,q)_{0,\infty},p,\ep,n)$ for the minimal cardinality of a $(p,\ep,n)$-spanning set for $(r,q)_{0,\infty}$.
\end{definition}

The existence of the finite maximal and minimal cardinalities in Definition~\ref{d:4.1} is guaranteed by Lemma~\ref{l:3.2}, because $(Z,\V)$ is a compact uniform space, $(Z^n,\V^n)$ is compact uniform as well, and $p^*_n\in\V^n$. Indeed, since $r_j$, $q_j$, $j\in\N\cup\{0\}$, are continuous maps between uniform spaces, they are also uniformly continuous. Also the canonical projections $\pi_i:Z^n\to Z$, $i=1,\ldots,n$, are uniformly continuous. Hence, $p(r_i\circ \pi_i(\cdot),r_i\circ \pi_i(\cdot)), p(q_i\circ \pi_i(\cdot),q_i\circ \pi_i(\cdot))\in\V^n$, $i=1,\ldots,n$, and also $p^*_n\in\V$, because it is a maximum of these pseudometrics.

Observe that if $r_j:Z\to X$, $j\in\N\cup\{0\}$, are continuous proper surjections and $X$ is compact, so must be $Z$. Furthermore, since the graphs $\Gamma_{r_j}$ of $r_j$, $j\in\N\cup\{0\}$, are closed in $Z\times X$, so must be the graphs $\Gamma_{r^{-1}_j}$ of (multivalued) $r^{-1}_j$, $j\in\N\cup\{0\}$, in $X\times Z$. Consequently, the multivalued maps $r^{-1}_j$ are, according to Proposition~\ref{p:2.1}, u.s.c. with compact values, i.e. $r^{-1}_j:X\to\KK(Z)$, $j\in\N\cup\{0\}$. The same is obviously true for the compositions $\vp_j=q_j\circ r^{-1}_j$, $j\in\N\cup\{0\}$, as well as for $\vp_{j+k}\circ\ldots\circ\vp_j = q_{j+k}\circ r^{-1}_{j+k} \circ \ldots \circ q_j\circ r^{-1}_j$, $j\in\N\cup\{0\}$, $k\in\N$, provided $q_j$, $j\in\N\cup\{0\}$, are continuous.

\begin{lemma}\label{l:4.1}
Let $X$, $Z$ be compact uniform spaces and $\vp_j=q_j\circ r_j^{-1}:X\multimap X$, $j\in\N\cup\{0\}$, where $q_j,r_j: Z\to X$, $j\in\N\cup\{0\}$, are single-valued continuous maps such that $r_j$, $j\in\N\cup\{0\}$, are surjections. Then the equality
\begin{align}\label{eq:1}
s(\vp_{0,\infty},p,\ep,n+1) = s((r,q)_{0,\infty},p,\ep,n)
\end{align}
holds for the maximal cardinalities of $(p,\ep,n+1)$-separated $(n+1)$-orbits for $\vp_{0,\infty}= (q_j\circ r_j^{-1})_{j=0}^{\infty}$ in Definition~\ref{d:3.2} and $(p,\ep,n)$-separated $n$-orbits of coincidences for $(r,q)_{0,\infty}$ in Definition~\ref{d:4.1}.
\end{lemma}
\begin{proof}
We construct the surjective mapping $\Phi:\coinn((r,q)_{0,\infty})\to \orb{n+1}(\vp_{0,\infty})$ satisfying
\begin{equation*}
p^*_{n}(\overline{z},\overline{z}') = p_{n+1}(\Phi(\overline{z}),\Phi(\overline{z}')), \mbox{ for every $\overline{z},\overline{z}'\in \coinn((r,q)_{0,\infty})$}.
\end{equation*}
Subsequently, the application of Lemma~\ref{l:sep} leads to a desired conclusion.

Coming back to the construction of $\Phi$, let $\overline{z}=(z_0,\ldots,z_{n-1})\in \coinn((r,q)_{0,\infty})$. Then $\overline{x}=(r_0(z_0), \ldots, r_{n-1}(z_{n-1}), q_{n-1}(z_{n-1}))\in \orb{n+1}(\vp_{0,\infty})$, because
\begin{align*}
\vp_i(r_i(z_i)) &= q_i \circ r_i^{-1}(r_i(z_i)) \in q_i(z_i) = r_{i+1}(z_{i+1}) \mbox{, for $i=0,\ldots,n-2$,}\\
\vp_{n-1}(r_{n-1}(z_{n-1})) &= q_{n-1} \circ r_{n-1}^{-1}(r_{n-1}(z_{n-1}))\in q_{n-1}(z_{n-1}).
\end{align*}

We define $\Phi(\overline{z}):=\overline{x}$.

Let, on the other hand, $\overline{x}=(x_0,\ldots,x_{n})\in \orb{n+1}(\vp_{0,\infty})$. For every $i=0,\ldots,n-1$, it holds $x_{i+1}=\vp_i(x_i)$, i.e. $x_{i+1}=q_i\circ r_i^{-1}(x_i)$, which means that there exists $z_i\in Z$ such that $x_{i+1}=q_i(z_i)$ and $x_i=r_i(z_i)$. Thus, $\overline{z} = (z_0,\ldots,z_{n-1})\in \coinn((r,q)_{0,\infty})$, $\Phi(\overline{z})=\overline{x}$, and $\Phi$ is surjective.

Finally, let $\overline{z},\overline{z}'\in \coinn((r,q)_{0,\infty})$, where $\overline{z}=(z_0,\ldots,z_{n-1})$, $\overline{z}'=(z'_0,\ldots,z'_{n-1})$. The following straightforward verification concludes the proof:
\begin{align*}
&p_{n+1}(\Phi(\overline{z}),\Phi(\overline{z}'))\\
&= p_{n+1}((r_0(z_0), \ldots, r_{n-1}(z_{n-1}), q_{n-1}(z_{n-1})),(r_0(z'_0), \ldots, r_{n-1}(z'_{n-1}), q_{n-1}(z'_{n-1}))\\
&= \max\{\max_{i=0,\ldots,n-1}\{p(r_i(z_i),r_i(z'_i))\}; p(q_{n-1}(z_{n-1}),q_{n-1}(z'_{n-1}))\}\\
&=\max\{\max_{i=0,\ldots,n-1}\{p(r_i(z_i),r_i(z'_i))\}; \max_{i=0,\ldots,n-1} p(q_{i}(z_{i}),q_{i}(z'_{i}))\}\\
&= \max_{i=0,\ldots,n-1} \max\{p((r_i(z_i)),r_i(z'_i)), p(q_i(z_i),q_i(z'_i))\}\\
&= p_n^*(\overline{z},\overline{z}').
\end{align*}
\end{proof}


\begin{remark}\label{r:4.1}
In view of Lemma~\ref{l:4.1}, the equality
\begin{align}\label{eq:2}
h(\vp_{0,\infty}) = h((r,q)_{0,\infty})
\end{align}
is (under the above assumptions) satisfied for the parametric topological entropies of $\vp_{0,\infty}$ and $(r,q)_{0,\infty}$ such that $\vp_j=q_j\circ r^{-1}_j$, $j\in\N\cup\{0\}$.
\end{remark}

On the other hand, the calculation or at least estimation of $h((r,q)_{0,\infty})$ is in this particular case still a difficult task in general. That is why it seems to be natural to impose for this goal some further restrictions on the space $X$ or on the selected pairs $(r,q)_{0,\infty}$.

One possibility is to consider the $m$-dimensional torus, i.e. $X=\T^m=\R^m/\Z^m$, and to require additionally that the values of $r_j^{-1}(x)$, $j\in\N\cup\{0\}$, are $R_\delta$-sets for every $x\in\T^m$. Then the multivalued maps $\vp_j=q_j\circ r_j^{-1}$, $j\in\N\cup\{0\}$, become admissible in the sense of G\'{o}rniewicz (see e.g. \cite[Chapter I.4]{AG}), and so we are able to define for them the Nielsen and Lefschetz numbers (see \cite{AG,AJ}).

In \cite{AL3}, we have derived in their terms the following Ivanov-type inequality for the lower estimate of $h((r,q)_{0,\infty})$ on $\T^m$.

\begin{prop}[cf. {\cite[Proposition 5]{AL3}}]\label{p:4.1}
Let the sequence $\vp_{0,\infty}$ of admissible maps $\vp_j:\T^m\to\KK(\T^m)$, $j\in\N\cup\{0\}$, be determined by the family $(r,q)_{0,\infty}$ of admissible selected pairs $(r_j,q_j)_{j=0}^\infty$, i.e. $\vp_j=q_j\circ r_j^{-1}$, $j\in\N\cup\{0\}$, where $r_j:Z\Rightarrow\T^m$, $j\in\N\cup\{0\}$, are equi-continuous Vietoris maps such that the values $r_j^{-1}(x)$ are $R_\delta$-maps, for every $x\in\T^m$, and $q_j:Z\to\T^m$, $j\in\N\cup\{0\}$, are equi-continuous maps. Then the inequality
\begin{align}
h((r,q)_{0,\infty}) \geq \log \max\left\{1,\limsup_{n\to\infty} \left| L(\vp_{n-1}\circ\ldots\circ\vp_0)\right|^{\frac1{n}}\right\},
\end{align}\label{eq:3}
holds for the parametric topological entropy $h((r,q)_{0,\infty})$ of $(r,q)_{0,\infty}$, defined in Definition~\ref{d:4.1}, where $L(\vp_{n-1}\circ\ldots\circ\vp_0)$ stands for the Lefschetz number of the composition $\vp_{n-1}\circ\ldots\circ\vp_0$, $n\in\N$ (for its definition, see e.g. \cite[Chapter I.6]{AG}).
\end{prop}

\begin{remark}
Since, under the assumptions of Proposition~\ref{p:4.1}, the admissible maps $\vp_j:\T^m\to\KK(\T^m)$ can be approximated with an arbitrary accuracy by single-valued continuous maps $f_j:\T^m\to\T^m$, which are admissibly homotopic with $\vp_j$, i.e. $\vp_j\simeq f_j$, $j\in\N\cup\{0\}$, $j\in\N\cup\{0\}$, and subsequently
\begin{equation*}
L(\vp_{n-1}\circ\ldots\circ\vp_0) = L(f_{n-1}\circ\ldots \circ f_0), n\in\N,
\end{equation*}
the inequality \eqref{eq:3} can be rewritten into the form
\begin{align}\label{eq:4}
h((r,q)_{0,\infty}) \geq \log \max\left\{1,\limsup_{n\to\infty} \left| L(f_{n-1}\circ\ldots \circ f_0)\right|^{\frac1{n}}\right\},
\end{align}
which is easier for calculations.
\end{remark}

On the circle $S^1 = \R/\Z$ (i.e. for $m=1$), the inequality \eqref{eq:4} takes the following simple form.
\begin{corollary}\label{c:4.1}
For $m=1$, the inequality
\begin{align}\label{eq:5}
h((r,q)_{0,\infty}) &\geq \log \max\left\{1,\limsup_{n\to\infty} \left| 1 - \prod_{j=0}^{n-1} \deg(f_j)\right|^{\frac1{n}}\right\}\\
&= \log \max \left\{1,\limsup_{n\to\infty} \left| 1 - \prod_{j=0}^{n-1} [ \hat{f}_j(1)-\hat{f}_j(0)]\right|^{\frac1{n}}\right\}\nonumber
\end{align}
holds for the parametric topological entropy $h((r,q)_{0,\infty})$ on the circle $S^1=\R/\Z$, where $\deg(f_j) = \hat{f}_j(1)-\hat{f}_j(0)$ denotes the topological degree of $f_j$ on $S^1$ and $\hat{f}_j$ are the lifts of the single-valued continuous approximations of $f_j$, $j\in\N\cup\{0\}$.
\end{corollary}

\begin{remark}
If $\vp = (r,q) = (r_j,q_j)$, for all $j\in\N\cup\{0\}$, and the single-valued continuous approximation $f\simeq\vp$ has degree $d$, i.e. $\deg(f) = \hat{f}(1) -\hat{f}(0) = d$, then the inequality \eqref{eq:5} reduces to
\begin{equation*}
h((r,q)) \geq \log \max\{1,|d|\} = \log\max \{1,|\hat{f}(1)-\hat{f}(0)|\}.
\end{equation*}
\end{remark}

Staying in dimension one for $X=[0,1]$, we can apply the result in \cite[Corollary 5.5]{AM2} as follows.

\begin{prop}\label{p:4.2}
Let $q_j,r_j:[0,1]\to[0,1]$, $j=0,1,\ldots k-1$, and $\tilde{q}, \tilde{r}:[0,1]\to[0,1]$ be single-valued continuous surjections such that
\begin{align}\label{eq:6}
(q_{k-1}\circ r^{-1}_{k-1})\circ \ldots \circ (q_0\circ r^{-1}_0) \subset \tilde{q} \circ \tilde{r}^{-1} \mbox{ on $[0,1]$,}
\end{align}
or
\begin{align}\label{eq:7}
(q_{k-1}\circ r^{-1}_{k-1})\circ \ldots \circ (q_0\circ r^{-1}_0) \supset \tilde{q} \circ \tilde{r}^{-1} \mbox{ on $[0,1]$.}
\end{align}

If $\tilde{q}$, $\tilde{r}$ are piecewise monotone and strongly commutative, i.e. $\tilde{q}\circ\tilde{r}^{-1} = \tilde{r}^{-1} \circ \tilde{q}$, then the inequalities
\begin{align}\label{eq:8}
h((q_{k-1}\circ r^{-1}_{k-1})\circ \ldots \circ (q_0\circ r^{-1}_0)) \leq \max\{h(\tilde{q}),h(\tilde{r})\},
\end{align}
resp.
\begin{align}\label{eq:9}
h((q_{k-1}\circ r^{-1}_{k-1})\circ \ldots \circ (q_0\circ r^{-1}_0)) \geq \max\{h(\tilde{q}),h(\tilde{r})\},
\end{align}
hold for the (nonparametric case of) topological entropy $h$ in the sense of Definition~\ref{d:4.1}.

If the sequence $(r,q)_{0,\infty} = \{q_j\circ r_j^{-1}\}_{j=0}^{\infty}$ consists of periodically repeated selected pairs with period $k$, i.e. $(r,q)_{0,\infty}=\{\overline{q_j\circ r^{-1}_j}\}_{j=0}^{k-1}$, where $q_{nk+j}=q_j$ and $r_{nk+j}=r_j$, for $j=0,1,\ldots, k-1$, $n\in\N\cup\{0\}$, then still

\begin{align}\label{eq:10}
h((r,q)_{0,\infty}) \leq \max\{h(\tilde{q}),h(\tilde{r})\} \mbox{ holds (under \eqref{eq:6})},
\end{align}
resp.
\begin{align}\label{eq:11}
h((r,q)_{0,\infty} \geq \frac1{k}\max\{h(\tilde{q}),h(\tilde{r})\} \mbox{ holds (under \eqref{eq:7})},
\end{align}
for the parametric topological entropy $h((r,q)_{0,\infty})$ of $(r,q)_{0,\infty}$ in the sense of Definition~\ref{d:4.1}.
\end{prop}

\begin{proof}
According to the (nonparametric versions of) Remarks~\ref{r:3.4} and \ref{r:4.1}, we have

\begin{align*}
h((q_{k-1}\circ r^{-1}_{k-1}\circ\ldots\circ(q_0\circ r_0^{-1}))\leq h(\tilde{q}\circ \tilde{r}^{-1}) \mbox{, for \eqref{eq:6}},
\end{align*}
resp.
\begin{align*}
h((q_{k-1}\circ r^{-1}_{k-1}\circ\ldots\circ(q_0\circ r_0^{-1}))\geq h(\tilde{q}\circ \tilde{r}^{-1}) \mbox{, for  \eqref{eq:7}}.
\end{align*}

Since, under the assumptions imposed on $\tilde{q}$, $\tilde{r}$, the equality
\begin{align}\label{eq:12}
h(\tilde{q}\circ\tilde{r}^{-1}) = \max\{h(\tilde{q}),h(\tilde{r})\}
\end{align}
is satisfied by means of \cite[Corollary 5.5]{AM2}, we arrive at the inequalities \eqref{eq:8}, resp. \eqref{eq:9}.

If the sequence $(r,q)_{0,\infty}$ consists of periodically repeated selected pairs $(r_j,q_j)$, $j\in\{0,1,\ldots,k-1\}$ with period $k$, then, according to Theorem~\ref{t:3.2} (cf. also \eqref{eq:2}), we obtain
\begin{align*}
h((r,q)_{0,\infty}) &\leq h((r,q)^{[k]}_{0,\infty}) = h((q_{k-1}\circ r^{-1}_{k-1})\circ \ldots \circ (q_0\circ r^{-1}_0))\\
&\leq \max\{h(\tilde{q}),h(\tilde{r})\}, \mbox{ for \eqref{eq:6} (i.e. \eqref{eq:10}),}
\end{align*}
resp.
\begin{align*}
h((r,q)_{0,\infty}) &\geq \frac1{k} h((r,q)^{[k]}_{0,\infty}) = \frac1{k} h((q_{k-1}\circ r^{-1}_{k-1})\circ \ldots \circ (q_0\circ r^{-1}_0))\\
&\geq \frac1{k}\max\{h(\tilde{q}),h(\tilde{r})\}, \mbox{ for \eqref{eq:7} (i.e. \eqref{eq:11}).}
\end{align*}

\end{proof}

Observe that, for $k=1$, we have $q_0=\tilde{q}$, $r_0=\tilde{r}$, which refers trivially to the original result \eqref{eq:12} in \cite{AM2}. Moreover, $h((q_0\circ r_0^{-1})^n) = n \cdot h(q_0\circ r^{-1}_0) = n \cdot \max\{h(q_0),h(r_0)\}$, $n\in\N$ (cf. \cite[Corollary 5.6]{AM1}).

For $q_j = \id\restr{[0,1]}$, $j=0,1,\ldots,k-1$, we can take $\tilde{q}=\id\restr{[0,1]}$, $(r_0\circ\ldots\circ r_{k-1})=\tilde{r}$. Since $r_j$, $j=0,1,\ldots,k-1$, are by the hypothesis piecewise monotone surjections, so must be then $\tilde{r}$, and obviously $\id\restr{[0,1]} \circ\tilde{r}^{-1} = \tilde{r}^{-1} \circ \id\restr{[0,1]}$. Thus, in view of \cite[Corollary~5.5]{AM2}, we get
\begin{align*}
h(r^{-1}_{k-1}\circ\ldots \circ r^{-1}_0) &= h((r_0\circ \ldots r_{k-1})^{-1}) = \max\{h(r_0\circ\ldots\circ r_{k-1}), h(\id\restr{[0,1]})\}\\
 &= h(r_0\circ \ldots\circ r_{k-1}),
\end{align*}
which improves both \eqref{eq:8} and \eqref{eq:9}. On the other hand, it is only a particular case of the result in \cite[Corollary 3.6]{KT}, which holds for general u.s.c. surjections.

For $k>1$ and $q_j\neq\id\restr{[0,1]}$, for at least some $j\in\{0,\ldots,k-1\}$, we can give the following illustrative example inspired by \cite[Proposition 3.3]{AM1}, dealing with symmetric tent maps.

\begin{example}
In accordance with \cite{AM1}, we define for $k\geq 2$ the \emph{symmetric $k$-tent map} $T_k:[0,1]\to[0,1]$ as a map with critical points $\{\frac{i}{k}: 0\leq i \leq k\}$ such that $T_k(\frac{i}{k}) = 0$, for even $i$, and $T_k(\frac{i}{k}) = 1$, for odd $i$, and such that $T_k\restr{[\frac{i}{k},\frac{i+1}{k}]}$ is linear, for every $i\in\{0,\ldots,n-1\}$.

These maps satisfy the following relations:
\begin{align*}
&T_k\circ T_l = T_l \circ T_k, \mbox{ for $k,l\geq 2$, (commutativity)},\\
&T_k \circ T_l = T_{kl}, \mbox{ for $k,l\geq 2$, (multiplicity)},\\
&T_k^{-1} \circ T_k \supset \id\restr{[0,1]}, T_k\circ T^{-1}_k = \id\restr{[0,1]}, \mbox{ for $k\geq 2$ (see e.g. \cite[Proposition I.3.1]{AG}).}
\end{align*}

Furthermore, $T_k$ and $T_l$, $k,l\geq 2$, strongly commute, i.e. $T_k\circ T_l^{-1}=T^{-1}_l \circ T_k$, $k,l\geq 2$, if and only if $k$ and $l$ are relatively prime (see \cite[Proposition~3.3]{AM1}).

Now taking, for instance, $k=2$ and $q_1=T_3$, $q_0=T_4$, $r_1=T_2$, $r_0=T_7$, the composition on the left-hand side of \eqref{eq:7} takes the form
\begin{align*}
(T_3\circ T_2^{-1}) \circ (T_4\circ T_7^{-1}) = T_3\circ T^{-1}_2 \circ T_2^2\circ T^{-1}_7 \supset T_3\circ \id\restr{[0,1]} \circ T_2\circ T^{-1}_7 = T_6\circ T^{-1}_7
\end{align*}
with $\tilde{q}=T_6$, $\tilde{r}=T_7$.

Hence, applying Proposition~\ref{p:4.2}, we get the inequality \eqref{eq:9} as
\begin{equation*}
h((T_3\circ T_2^{-1})\circ(T_4\circ T_7^{-1}))\geq \max\{h(T_6),h(T_7)\} = h(T_7) = \log 7.
\end{equation*}

For the inequality \eqref{eq:8}, it is enough to take, for instance, $k=2$ and $q_1=T_4$, $q_0=T_4$, $r_1=T_2$, $r_0=T_7$, because then we obtain the inclusion \eqref{eq:6} in the form
\begin{align*}
(T_4\circ T_2^{-1})\circ(T_4\circ T_7^{-1}) = T_2\circ\id\restr{[0,1]}\circ T_4\circ T_7^{-1} = T_8\circ T_7^{-1}
\end{align*}	
with $\tilde{q}=T_8$, $\tilde{r}=T_7$. Subsequently, by means of \eqref{eq:8}, we get
\begin{equation*}
h((T_4\circ T_2^{-1})\circ(T_4\circ T_7^{-1}))= \max\{h(T_8),h(T_7)\} = h(T_8) = \log 8.
\end{equation*}

Now, if the sequence $(r,q)_{0,\infty}$ consists of $2$-periodically repeated (i.e. $k=2$) selected pairs $(T_2,T_3)$, $(T_7,T_4)$, then we obtain by \eqref{eq:11} the lower estimate (cf. \eqref{eq:9})
\begin{align*}
h((r,q)_{0,\infty})\geq \frac12 h((T_3\circ T_2^{-1})\circ(T_4\circ T_7^{-1})) \geq \frac12 h(T_7) = \frac12 \log 7 = \log \sqrt{7}.
\end{align*}

If the sequence $(r,q)_{0,\infty}$ consists of $2$-periodically repeated (i.e. $k=2$) selected pairs $(T_2,T_4)$, $(T_7,T_4)$, then we obtain by \eqref{eq:10} the upper estimate (cf. \eqref{eq:8})
\begin{align*}
h((r,q)_{0,\infty})\leq  h((T_4\circ T_2^{-1})\circ(T_4\circ T_7^{-1})) = h(T_8) = \log 8.
\end{align*}

If, in particular, $k$ and $l$ with $k\geq l\geq 2$ are relatively prime integers, then $h((T_k\circ T_l^{-1})^n)=n \cdot h(T_k\circ T_l^{-1}) = n \cdot \max\{ h(T_k),h(T_l)\} = n \cdot h(T_k) = n \log k,$ for every $n\in\N$ (cf. \cite[Corollary 5.6]{AM1}).

\end{example}

Coming back to the scheme
\begin{equation*}
X\overset{r_j}{\leftarrow}Z\overset{q_j}{\rightarrow} X, \, j\in\N\cup\{0\},
\end{equation*}
but this time with $q_j=\id\restr{Z}$, $j\in\N\cup\{0\}$, requires to put $Z=X$, which leads to the parametric preimage entropy $h((r,\id\restr{X})_{0,\infty})$ of the sequence $(r,\id\restr{X})_{0,\infty}$ or, equivalently, to the parametric topological entropy $h(r^{-1}_{0,\infty})$ of the sequence $r^{-1}_{0,\infty} = \{r^{-1}_j\}_{j=0}^{\infty}$ of multivalued maps $r^{-1}_j:X\multimap X$, $j\in\N\cup\{0\}$, (as a particular case of our former study in the foregoing Section~\ref{s:3}), provided $r_j:X\to X$, $j\in\N\cup\{0\}$, are surjections on a compact Hausdorff space $X$.

Furthermore, observe that since then
\begin{equation*}
\orbn(r^{-1}_{0,\infty}) = \{x_0\in X\} \oplus \coin{n-1}((r,\id\restr{X})_{0,\infty}), n\geq 2,
\end{equation*}
where
\begin{align*}
&\orbn(r^{-1}_{0,\infty}) = \{(x_0,x_1,\ldots,x_{n-1}) \in X^n: r_i^{-1}(x_i)=x_{i+1}, i=0,1,\ldots,n-1\},\\
&\coin{n-1}((r,\id\restr{X})_{0,\infty}) = \{(x_1,\ldots,x_{n-1})\in X^{n-1}: r_{i-1}(x_i)=x_{i-1}, i=1,\ldots,n-1\},
\end{align*}
the equality (cf. \eqref{eq:1})
\begin{equation*}
s(r^{-1}_{0,\infty},p,\ep,n) = s((r,\id\restr{X})_{0,\infty},p,\ep,n-1), \, n\geq 2,
\end{equation*}
holds for the maximal cardinalities of $(p,\ep,n)$-separated $n$-orbits of $r^{-1}_{0,\infty}$ and $(p,\ep,n-1)$-separated $(n-1)$-orbits of coincidences of $(r,\id\restr{X})_{0,\infty}$.

Hence, the equality
\begin{align}\label{eq:13}
h(r^{-1}_{0,\infty}) = h((r,\id\restr{X})_{0,\infty})
\end{align}
holds for the topological entropies in the sense of Definitions~\ref{d:3.8} and~\ref{d:4.1}, as claimed.

Since the multivalued maps $r^{-1}_j$, $j\in\N\cup\{0\}$, in the sequence $r^{-1}_{0,\infty}$ are the inversions of single-valued maps $r_j$, $j\in\N\cup\{0\}$, it can allow us to detect some properties of the preimage entropy $h(r^{-1}_{0,\infty})$ by means of $h(r_{0,\infty})$.

\begin{lemma}\label{l:4.2}
Let $X$ be a compact metric space and let the system $r_{0,\infty}$ consist of periodically repeated pairs of single-valued continuous surjections, i.e. $r_{0,\infty} = (\overline{r_0,r_1})$, where $r_j=r_{2n+j}$, for $j=0,1$, $n\in\N\cup\{0\}$. Then
\begin{align}\label{eq:14}
h(r_{0,\infty}) \leq h(r^{-1}_{0,\infty}) \leq 2h(r_{0,\infty}),
\end{align}
holds for the topological entropies $h(r_{0,\infty})$ of $r_{0,\infty}$ and $h(r^{-1}_{0,\infty})$ of $r^{-1}_{0,\infty}$ in the sense of Definition~\ref{d:3.8}.
\end{lemma}

\begin{proof}
According to Theorem~\ref{t:3.2}, we get
\begin{align*}
\frac12 h(r^{-1}_1\circ r^{-1}_0) = \frac12 ((r^{-1}_{0,\infty})^{[2]}) \leq h(r^{-1}_{0,\infty}) \leq h((r^{-1}_{0,\infty})^{[2]}) = h(r^{-1}_1\circ r^{-1}_0).
\end{align*}

Since, by means of \cite[Corollary 3.6]{KT},
\begin{equation*}
h((r_0\circ r_1)^{-1}) = h(r_0\circ r_1),
\end{equation*}
by virtue of \cite[Theorem A]{KS},
\begin{equation*}
h(r_0\circ r_1) = h(r_1 \circ r_0) = h((r_{0,\infty})^{[2]})
\end{equation*}
and, in view of \cite[Lemma 4.3]{KS},
\begin{equation*}
h((r_{0,\infty})^{[2]}) = 2h(r_{0,\infty}),
\end{equation*}
we arrive at the equality \eqref{eq:13}.

\end{proof}

Thus, in particular, if $h(r_{0,\infty})=0$ or $h(r_{0,\infty})<\infty$ or $h(r_{0,\infty})=\infty$, then respectively $h(r^{-1}_{0,\infty})=0$ or $h(r^{-1}_{0,\infty})<\infty$ or $h(r^{-1}_{0,\infty})=\infty$.

\begin{remark}
Since we used in the proof of Lemma~\ref{l:4.2} the result in \cite[Theorem A]{KS}, it is impossible, according to \cite[Remark after Corollary 5.2]{KS} to consider in Lemma~\ref{l:4.2}, for $k>2$, periodically repeated $k$-tuples of single-valued surjections, i.e. $r_{0,\infty}=(\overline{r_0,\ldots,r_{n-1}})$, where $r_j=r_{kn+j}$, for $j=0,1,\ldots,k-1$, $n\in\N\cup\{0\}$, in order to satisfy \eqref{eq:13}.
\end{remark}

\begin{example}\label{e:4.2}
Let $r_{0,\infty}$ consist of periodically repeated pairs of single-valued continuous surjections $r_j:S^1\to S^1$ such that $\deg(r_j)=d_j$, $j=0,1$.

Assuming that (cf. \eqref{eq:5}) 
\begin{equation*}
|d_0 d_1|>1,
\end{equation*}
we obtain, in view of Corollary~\ref{c:4.1} and \eqref{eq:13} in Lemma~\ref{l:4.2}, that 
\begin{align*}
h(r^{-1}_{0,\infty}) &\geq h((r_{0,\infty}) \geq \log\left(\limsup_{n\to\infty} | 1 - \prod_{j=0}^{n-1} d_j|^{\frac1{n}}\right) = \log\max\{\sqrt{|d_0d_1|},|d_0|,|d_1|\}\\
&= \log \max\{ |d_0|,|d_1|\}>0.
\end{align*}

In view of \eqref{eq:2}, this result can be interpreted as the lower estimate of the preimage entropy $h((r,\id\restr{S^1})_{0,\infty})$, dealing with the sequence $r_{0,\infty}^{-1}$ of multivalued u.s.c. maps $r_j^{-1}:S^1\to \KK(S^1)$, $j\in\N\cup\{0\}$.

\end{example}

Nevertheless, if $\vp_{0,\infty}$ consists of multivalued maps which cannot be determined by single-valued maps, then for the lower and upper estimates of $h(\vp_{0,\infty})$ we need different types of tools. Some of them will be indicated in the next section.

\vspace{10pt}

This section will be concluded by another definition of parametric topological entropy on orbits of coincidences which does not coincide, in the single-valued case, with any standard definition in \cite{AKM, Bo, Di, KS} (see Example~\ref{e:4.3} below).

Hence, let $X=(X,\U)$ and $Z=(Z,\V)$ be compact uniform spaces and  $(r,q)_{0,\infty}$ be the sequence of selected pairs $(r_j,q_j)_{j=0}^{\infty}$, where $X\overset{r_j}{\leftarrow}Z\overset{q_j}{\rightarrow} X$, $j\in\N\cup\{0\}$, are single-valued maps such that $r_j$, $j\in\N\cup\{0\}$, are surjections.

Let for this time $p\in\V$, $n$ be a positive integer and $\ep>0$. We say that the subset $S\subset\coinn((r,q)_{0,\infty})$ is \emph{$(p,\ep,n)_Z$-separated for $(r,q)_{0,\infty}$}, resp. the subset $R\subset\coinn((r,q)_{0,\infty})$ is \emph{$(p,\ep,n)_Z$-spanning for for $(r,q)_{0,\infty}$}, if it is a $(p_n,\ep)$-separated subset of $\coinn((r,q)_{0,\infty})$, resp. a $(p_n,\ep)$-spanning set in $\coinn((r,q)_{0,\infty})$, in the sense of Definition~\ref{d:3.1}, where
\begin{align}\label{eq:15}
p_n((z_0,\ldots,z_{n-1}),(z'_0,\ldots,z'_{n-1})):= \max\{p(z_i,z'_i):i=0,\ldots,n-1\},
\end{align}
for $(z_0,\ldots,z_{n-1}),(z'_0,\ldots,z'_{n-1})\in Z^n$.

\begin{definition}\label{d:4.2}
The \emph{parametric topological entropy} on orbits of coincidences $h_Z((r,q)_{0,\infty})$ for a family $(r,q)_{0,\infty}$ of selected pairs $(r_j,q_j)_{j=0}^{\infty}$ takes the form
\begin{align*}
h_Z((r,q)_{0,\infty})&:= \sup_{\ep>0,p\in\U} \limsup_{n\to\infty}\frac1{n}\log s_Z((r,q)_{0,\infty},p,\ep,n) \\
&= \sup_{\ep>0,p\in\U} \limsup_{n\to\infty}\frac1{n}\log r_Z((r,q)_{0,\infty},p,\ep,n),
\end{align*}
where $s_Z((r,q)_{0,\infty},p,\ep,n)$ stands for the maximal cardinality of a $(p,\ep,n)_Z$-separated set for $(r,q)_{0,\infty}$, and $r_Z((r,q)_{0,\infty},p,\ep,n)$ for the minimal cardinality of a $(p,\ep,n)_Z$-spanning set for $(r,q)_{0,\infty}$.
\end{definition}

The existence of finite maximal and minimal cardinalities in Definition~\ref{d:4.2} is guaranteed by Lemma~\ref{l:3.2}, because $p_n\in\V^n$ (see the reasoning below Definition~\ref{d:3.2}). The equality of the separated and spanning variants arises from Lemma~\ref{l:sep-span} (cf. Definitions~\ref{d:hKT} and~\ref{d:4.1}).

One can readily check that the maximal cardinalities $s((r,q)_{0,\infty}, p,\ep,n)$ and $s_Z((r,q)_{0,\infty},p,\ep,n)$ in Definitions~\ref{d:4.1} and \ref{d:4.2} can significantly differ each to other. For instance, if $p$, $p_n$ in \eqref{eq:15} become metrics, then obviously the strict inequality
\begin{equation*}
s((r,q)_{0,\infty},p,\ep,n) < s_Z((r,q)_{0,\infty},p,\ep,n), n\in\N,
\end{equation*}
can occur like in the following simple illustrative example.

\begin{example}\label{e:4.3}
%

Let $X=\{0\}$ be a singleton and $Z=\{z_0,z_1\}$, where $z_0,z_1\in \{ z\in \mathcal{C}([0,1]): z(0)=z(1)=0\}$, are two distinct single-valued continuous functions whose values vanish at $0$ and $1$. Assuming that $r_j(z_0)=r_j(z_1)=q_j(z_0) = q_j(z_1)=0$, $j\in\N\cup\{0\}$, we have $\orbn((q\circ r^{-1})_{0,\infty}) = \{(0,\ldots,0)\}$, while $\coinn((r,q)_{0,\infty}) = \{(u_0,\ldots,u_{n-1}):u_i\in\{z_0,z_1\}, i=0,\ldots,n-1\}$. We consider the uniformity on $Z$ which is induced by the usual supremum metric on $\mathcal{C}([0,1])$. To compute $h((r,q)_{0,\infty})$ and $h_Z((r,q)_{0,\infty})$, it is enough to take this metric as $p$, and
\begin{equation*}
p_n(u,v)=\max_{i=0,\ldots,n-1} p(u_i,v_i) = \max_{i=0,\ldots,n-1} \max_{t\in[0,1]} |u_i(t)-v_i(t)|,
\end{equation*}
where $u=(u_0,\ldots,u_{n-1})$, $v=(v_0,\ldots,v_{n-1})$, $u,v\in Z^n$.

Thus, for a sufficiently small $\ep>0$, $s((r,q)_{0,\infty}),p,\ep,n)=1$, and so $h((r,q)_{0,\infty}) = 0$ in the sense of Definition~\ref{d:4.1}, but $s_Z((r,q)_{0,\infty}),p,\ep,n)=2^n$, and so $h_Z((r,q)_{0,\infty})=\log 2$.

Obviously, if $X=\{0\}$ and $Z = \{z\in \mathcal{C}([0,1]):z(0)=z(1)=0\}$, then $h((r,q)_{0,\infty})=0$, but $h_Z((r,q)_{0,\infty})=\infty$ (cf. \cite[Theorem 3]{RS}).

\end{example}

Observe that since in fact $\vp_{0,\infty}= (q\circ r^{-1})_{0,\infty} = \{0,\ldots\}$, i.e. the sequence of single-valued zero constants, Definition~\ref{d:4.2} has no standard single-valued analog in \cite{AKM,Bo,Di,KS}. 

Let us conclude that although Definition~\ref{d:4.1} is only a particular case of Definition~\ref{d:4.2} and, in view of Remark~\ref{r:4.1}, also of Definition~\ref{d:hsepspan}, it seems to be more suitable for applications, especially because of Proposition~\ref{p:4.1} (cf. also Example~\ref{e:4.2}). Moreover, it can serve (unlike Example~\ref{e:4.3}) as a tool for the lower estimate of the parametric topological entropies in the sense of Definitions~\ref{d:hsepspan} and~\ref{d:4.2}.

\section{Some further possibilities}

In this section, the pointwise entropies $\hp$, $\hm$ and the inverse image entropy $\hi$, introduced in 1995 by Hurley (cf. \cite{Hu}) for the inversions of single-valued continuous maps in compact metric spaces, will be parametrized and examined for arbitrary nonautonomous multivalued maps in compact uniform spaces.

Hence, let $X=(X,\U)$ be a compact uniform space and $\vp_j:X\multimap X$, $j\in\N\cup\{0\}$, be a sequence $\vp_{0,\infty}$ of multivalued maps. Defining, for a given $x\in X$,
\begin{equation*}
\orbn(\vp_{0,\infty},x):= \{(x_0,x_1,\ldots,x_{n-1})\in\orbn(\vp_{0,\infty}): x_0=x\},
\end{equation*}
we obviously have
\begin{align}\label{eq:16}
\orbn(\vp_{0,\infty}) = \bigcup_{x\in X} \orbn(\vp_{0,\infty},x)\mbox{ and } \orbn(\vp_{0,\infty},x)\subset \orbn(\vp_{0,\infty}),
\end{align}
for every $x\in X$, $n\in\N$.

We can call again (as in Section~\ref{s:3}) $S\subset \orbn(\vp_{0,\infty},x)$ to be \emph{$(p,\ep,n,x)$-separated} for $\vp_{0,\infty}$, resp. $R\subset\orbn(\vp_{0,\infty},x)$ to be \emph{$(p,\ep,n,x)$-spanning} for $\vp_{0,\infty}$, in the sense of Definition~\ref{d:3.2}, if it is a $(p_n,\ep)$-separated subset, resp. a $(p_n,\ep)$-spanning set in $\orbn(\vp_{0,\infty},x)$, in the sense of Definition~\ref{d:3.1}, where $p_n$ is defined by \eqref{eq:0}.

\begin{definition}\label{d:5.1}
Let $X=(X,\U)$ be a compact uniform space and $\vp_j:X\multimap X$, $j\in\N\cup\{0\}$, be a sequence $\vp_{0,\infty}$ of multivalued maps. Denoting, for a given $x\in X$, by $s(\vp_{0,\infty},p,\ep,n,x)$ the largest cardinality of a $(p,\ep,n,x)$-separated set for $\vp_{0,\infty}$, resp. by $r(\vp_{0,\infty},p,\ep,n,x)$ the smallest cardinality of a $(p,\ep,n,x)$-spanning set for $\vp_{0,\infty}$ (it is possible by virtue of Lemma~\ref{l:3.2}).

The \emph{parametric pointwise entropy} $\hp(\vp_{0,\infty})$ of $\vp_{0,\infty}$ is defined (cf. Definition~\ref{d:hKT} and \eqref{eq:16}) to be
\begin{align*}
\hp(\vp_{0,\infty}) &:= \sup_{x\in X} \sup_{p\in\U,\ep>0} \limsup_{n\to\infty} \frac1{n} \log s(\vp_{0,\infty},p,\ep,n,x)\\
&= \sup_{x\in X} \sup_{p\in\U,\ep>0} \limsup_{n\to\infty} \frac1{n} \log r(\vp_{0,\infty},p,\ep,n,x).
\end{align*}

The \emph{parametric pointwise entropy} $\hm(\vp_{0,\infty})$ of $\vp_{0,\infty}$ is defined to be
\begin{align*}
\hm(\vp_{0,\infty}) &:= \sup_{p\in\U,\ep>0} \limsup_{n\to\infty} \frac1{n} \log \sup_{x\in X} s(\vp_{0,\infty},p,\ep,n,x)\\
&=\sup_{p\in\U,\ep>0} \limsup_{n\to\infty} \frac1{n} \log \sup_{x\in X} r(\vp_{0,\infty},p,\ep,n,x).
\end{align*}

\end{definition}

\begin{remark}\label{r:5.2}
Definition~\ref{d:5.1} generalizes, besides others, its analogs in \cite[Definition 2.3]{WZZ}, for autonomous u.s.c. maps in compact metric spaces, in \cite{HWZ,ZZH}, for the inversions of single-valued continuous maps in compact metric spaces, and in \cite{Hu,WZ}, for the inversions of autonomous single-valued continuous maps in compact metric spaces.
\end{remark}

\begin{definition}\label{d:5.2}
Let $X$ be a compact topological space and $\vp_j:X\multimap X$, $j\in\N\cup\{0\}$, be a sequence $\vp_{0,\infty}$ of multivalued maps. Let us recall that, for an open cover $\A$ of $X$ and $n\in X$, $n\geq 1$,
\begin{equation*}
\A^n= \{ U_0\times \ldots \times U_{n-1}: U_0, \ldots,U_{n-1}\in\A\}.
\end{equation*}

Let, for a given $x\in X$, $N(\vp_{0,\infty}, \A, n,x)$ be a minimal cardinality of a subcover of $\A^n$ covering $\orbn(\vp_{0,\infty},x)$.

Then the \emph{parametric pointwise entropy} $\hp^{\operatorname{top}}(\vp_{0,\infty})$ of $\vp_{0,\infty}$ can be defined as
\begin{equation*}
\hp^{\operatorname{top}}(\vp_{0,\infty}):= \sup_{x\in X} \sup_{\A} \limsup_{n\to\infty} \frac1{n} \log N(\vp_{0,\infty},\A,n,x).
\end{equation*}

The \emph{parametric pointwise entropy} $h^{\operatorname{top}}_m(\vp_{0,\infty})$ of $\vp_{0,\infty}$ can be defined as
\begin{equation*}
\hm^{\operatorname{top}}(\vp_{0,\infty}):=  \sup_{\A} \limsup_{n\to\infty} \frac1{n} \log \sup_{x\in X} N(\vp_{0,\infty},\A,n,x).
\end{equation*}
\end{definition}

\begin{remark}\label{r:5.1}
The compact topological space $X$ in Definition~\ref{d:5.2} need not be uniform, like in Definition~\ref{d:5.1}. If it is a compact Hausdorff space, then Definitions~\ref{d:5.1} and \ref{d:5.2} are equivalent, in view of Theorem~\ref{t:KT=AK} and \eqref{eq:16}, i.e.
\begin{align}\label{eq:17}
\hp(\vp_{0,\infty}) = \hp^{\operatorname{top}}(\vp_{0,\infty}) \mbox{ and } \hm(\vp_{0,\infty}) = \hm^{\operatorname{top}}(\vp_{0,\infty}).
\end{align}
\end{remark}

\begin{remark}\label{r:5.3}
Definition~\ref{d:5.2} generalizes, besides others, its analog in \cite{HWZ}, for the inversions of single-valued continuous maps in compact topological spaces and, because of \eqref{eq:17}, also those mentioned in Remark~\ref{r:5.2}, for the inversions of single-valued continuous maps in compact metric spaces. For single-valued maps, Definition~\ref{d:5.1} and \ref{d:5.2} have obviously no meaning.
\end{remark}

Now, define the pseudometric $p_n^b$ on $X$ as
\begin{align*}
p^b_n(x,y)&:= \pnH(\orbn(\vp_{0,\infty},x),\orbn(\vp_{0,\infty},y))\\
&=\max\{\sup_{\overline{x}\in\orbn(\vp_{0,\infty},x)} \inf_{\overline{y}\in\orbn(\vp_{0,\infty},y)} p_n(\overline{x},\overline{y}),\\
&\sup_{\overline{y}\in\orbn(\vp_{0,\infty},y)} \inf_{\overline{x}\in\orbn(\vp_{0,\infty},x)} p_n(\overline{y},\overline{x})\},\nonumber
\end{align*}
where $p_n\in\U^n$ is defined by \eqref{eq:0} and $\pnH$ along the lines of Proposition~\ref{p:hyper2}.


\begin{definition}\label{d:5.3}
Let $X=(X,\U)$ be a compact uniform space and $\vp_j:X\to\KK(X)$, $j\in\N\cup\{0\}$, be a sequence $\vp_{0,\infty}$ of multivalued maps. Denoting by $s(\vp_{0,\infty},p^b_n,\ep)$ the supremum of cardinalities of $(p^b_n,\ep)$-separated sets of $X$, resp. by $r(\vp_{0,\infty},p^b_n,\ep)$ the smallest cardinality of a $(p^b_n,\ep)$-spanning set in $X$, the \emph{parametric branch entropy} $\hi(\vp_{0,\infty})$ is defined to be (cf. Definition~\ref{d:hKT} and \eqref{eq:16})
\begin{align*}
\hi(\vp_{0,\infty})&:= \sup_{p\in\U,\ep>0} \limsup_{n\to\infty} \frac1{n} \log \fin(s(\vp_{0,\infty},p^b_n,\ep))\\
&= \sup_{p\in\U,\ep>0} \limsup_{n\to\infty} \frac1{n} \log \fin(r(\vp_{0,\infty},p^b_n,\ep)),
\end{align*}
where
\begin{equation*}
\fin(\kappa)=
\begin{cases}
&\kappa, \mbox{ if $\kappa$ is a finite cardinal number,}\\
&\infty, \mbox{ if $\kappa$ is an infinite cardinal number.}
\end{cases}
\end{equation*}
\end{definition}

\begin{remark}\label{r:5.4}
Let $(X,\U)$ be a compact uniform space and $\vp_j: X\to\KK(X)$, $j\in\N\cup\{0\}$, be a sequence $\vp_{0,\infty}$ of u.s.c. multivalued maps. Then we can omit the operator $\fin$ and use for the cardinalities the attribute maximal rather than supremal in Definition~\ref{d:5.3}. A natural way of doing it would be to show that $p_n^b\in\U$ and apply Lemma~\ref{l:3.2}. However, it is not clear whether $p_n^b$ is really an element of $\U$. 

Nevertheless, we can argue as follows. The space $(X^n,\U^n)$ is compact uniform and, by Proposition~\ref{p:hyper2}, $(\KK(X^n),(\U^n)^{\mathrm{H}})$ is also a compact uniform space.

Given $\ep>0$, $p\in\U$ and $n\in\N$, we can consider
\begin{equation*}
Y := \{ \orbn(\vp_{0,\infty},x): x\in X\},
\end{equation*}
which is a subset of $\KK(X^n)$ by Lemma~\ref{l:1}. Now, applying Lemma~\ref{l:3.2}, we can see that there exists a finite $(\pH_n,\ep)$-spanning set in $Y$ and also the cardinality of all $(\pH_n,\ep)$-separated subsets of $Y$ is uniformly bounded. Consequently, both the numbers $s(\vp_{0,\infty},p^b_n,\ep)$ and $r(\vp_{0,\infty},p^b_n,\ep)$ are finite.
\end{remark}

\begin{remark}\label{r:5.4b}
Definition~\ref{r:5.3} generalizes, besides others, its analogs in \cite[Definition 2.4]{WZZ}, for autonomous u.s.c. maps in compact metric spaces, in \cite{YWW,ZZH}, for the inversions of single-valued continuous maps in compact metric spaces and, in \cite{Hu,CN}, for the inversions of autonomous single-valued continuous maps in compact metric spaces. For single-valued maps, $h(\vp_{0,\infty})=\hi(\vp_{0,\infty})$ holds.
\end{remark}

The parametric entropies $\hp$, $\hm$, $\hi$ in Definitions~\ref{d:5.1}-\ref{d:5.3} are related to the parametric topological entropy $h$ in the sense of Definition~\ref{d:3.8} in the following way.

%
\begin{theorem}\label{t:5.1}
Let $X$ be a compact Hausdorff space and $\vp_{j}:X\multimap X$, $j\in\N\cup\{0\}$, be a sequence $\vp_{0,\infty}$ of multivalued maps. Then the inequalities
\begin{align}\label{eq:18}
\hp(\vp_{0,\infty}) &\leq \hm(\vp_{0,\infty}) \leq h(\vp_{0,\infty}),\\ \label{eq:18b}
h(\vp_{0,\infty}) &\leq \hm(\vp_{0,\infty}) + \hi(\vp_{0,\infty})
\end{align}
hold for the parametric topological entropy $h$ in the sense of Definition~\ref{d:3.8} and the parametric entropies $\hp$, $\hm$, $\hi$, in the sense of Definitions~\ref{d:5.1}-\ref{d:5.3}, respectively.
\end{theorem}

\begin{proof}
The inequalities \eqref{eq:18} follow straightforwardly from the definitions. It remains to prove \eqref{eq:18b}. 

Let $p\in\U$, $\ep>0$ and $S$ be the maximal $(p^b_n,\frac{\ep}3)$-separated subset of $X$ (see Remark~\ref{r:zorn}). For each $x\in S$, let $M(x)$ be the maximal $(p_n,\frac{\ep}3)$-separated set in $\orbn(\vp_{0,\infty},x)$. Let $M:= \bigcup_{x\in S} M(x)$. We will show that $M\subset \orbn(\vp_{0,\infty})$ is a $(p_n,\ep)$-spanning set in $\orbn(\vp_{0,\infty})$.

For any $\overline{y}=(y,y_1,\ldots,y_{n-1})\in\orbn(\vp_{0,\infty})$, there exists $x\in S$ such that $p^b_n(y,x)<\frac{\ep}3$ (otherwise, we would have a contradiction with the maximality of $S$). Employing the definition of $p^b_n$, there exists $\overline{x}=(x,x_1,\ldots,x_{n-1})\in\orbn(\vp_{0,\infty},x)$ such that $p^b_n(\overline{y},\overline{x})<\frac{\ep}3$.

Since $M(x)$ is the maximal $(p_n,\frac{\ep}3)$-separated set in $\orbn(\vp_{0,\infty},x)$, there is $\overline{x'}=(x,x'_1,\ldots,x'_{n-1})\in\orbn(\vp_{0,\infty},x)$ such that $p_n(\overline{x},\overline{x'})<\frac{\ep}3$. Hence, $p_n(\overline{y},\overline{x'})<\frac{2\ep}{3}<\ep$. We can conclude that $M(x)$ is a $(p_n,\ep)$-spanning set in $\orbn(\vp_{0,\infty},x)$. 

The cardinality of $M$ can be estimated as
\begin{align*}
&r(\vp_{0,\infty},p,\frac{\ep}{3},n)\leq \card M = \card \left( \bigcup_{x\in S} M(x)\right) \leq \sup_{x\in S} (\card(M(x))) \cdot \card(S)\\
&\leq \sup_{x\in S} r(\vp_{0,\infty},p,\frac{\ep}3,n,x)\cdot \card S \leq \sup_{x\in S} r(\vp_{0,\infty},p,\frac{\ep}3,n,x)\cdot s(\vp_{0,\infty}, p_n^b, \ep),
\end{align*}
and consequently
\begin{equation*}
r(\vp_{0,\infty},p,\frac{\ep}{3},n)\leq \sup_{x\in S} r(\vp_{0,\infty},p,\frac{\ep}3,n,x)\cdot \fin(s(\vp_{0,\infty}, p_n^b, \ep)).
\end{equation*}

The following computation concludes the proof, namely
\begin{align*}
h(\vp_{0,\infty}) &\leq \sup_{p\in\U,\ep>0} \limsup_{n\to\infty} \frac1{n} \log r(\vp_{0,\infty},p,\frac{\ep}{3},n)\\
&\leq \sup_{p\in\U,\ep>0} \limsup_{n\to\infty} \frac1{n} \log \sup_{x\in S} r(\vp_{0,\infty},p,\frac{\ep}3,n,x)\cdot \fin(s(\vp_{0,\infty}, p_n^b, \ep))\\
&\leq \sup_{p\in\U,\ep>0} \limsup_{n\to\infty} \frac1{n} \log \sup_{x\in S}r(\vp_{0,\infty},p,\frac{\ep}3,n,x) \\
&+ \sup_{p\in\U,\ep>0} \limsup_{n\to\infty} \frac1{n} \log \fin(s(\vp_{0,\infty}, p_n^b, \ep)) \\
&= \hm(\vp_{0,\infty}) + \hi(\vp_{0,\infty}).
\end{align*}
\end{proof}

\begin{remark}
The first version of \eqref{eq:18} and \eqref{eq:18b} was formulated by Hurley in \cite{Hu} for the inversions of autonomous single-valued continuous maps in compact metric spaces. For its extension to nonautonomous dynamical systems, see e.g. \cite{ZZH}. Recently, a multivalued version was established in \cite{WZZ} for autonomous u.s.c. maps in compact metric spaces. Theorem~\ref{t:5.1} therefore generalizes all these quoted results.
\end{remark}

\begin{remark}
If in Definitions~\ref{d:3.8} and \ref{d:5.1}-\ref{d:5.3}, the multivalued maps $\vp_j$, $j\in\N\cup\{0\}$, are specially admissible in the sense of G\'{o}rniewicz (for the definition, see Section~\ref{s:2}), then these maps as well as their finite compositions $\vp=\vp_{n-1}\circ\ldots\circ\vp_0$, $n\in\N$, can be always determined by admissible pairs $(r_j,q_j)$, $j\in\N\cup\{0\}$, and $(r,q)$ of single-valued continuous maps in order $\vp_j=q_j\circ r^{-1}_j$, $j\in\N\cup\{0\}$, and $\vp=q\circ r^{-1}$ respectively. If the space $X$ is still a compact absolute retract (e.g. a convex compact metric space), then these maps and their finite compositions possess fixed points (for more details, see e.g. \cite[Chapters I.4 and I.6]{AG}). The parametric topological entropies $h$, $\hp$, $\hm$, $\hi$ of these maps are then intimately related to the exponential growth of the number of fixed points under expanding compositions of maps.
\end{remark}

In order to obtain ``pointwise'' versions of Definition~\ref{d:4.2}, let us again consider compact uniform spaces $X=(X,\U)$, $Z=(Z,\V)$ and the sequence $(r,q)_{0,\infty}$ of selected pairs $(r_j,q_j)^\infty_{j=0}$, where $r_j: Z\to X$, $j\in\N\cup\{0\}$, are single-valued surjections and $q_j: Z\to X$, $j\in\N\cup\{0\}$, are single-valued maps.

Defining, for a given $z\in Z$,
\begin{equation*}
\coinn((r,q)_{0,\infty},z) = \{(z_0,z_1,\ldots,z_{n-1})\in\coinn((r,q)_{0,\infty}): z_0=z\},
\end{equation*}
we obviously have
\begin{align}\label{eq:19}
\coinn((r,q)_{0,\infty}) &:= \bigcup_{z\in Z} \coinn((r,q)_{0,\infty},z) \mbox{ and }\\
\coinn((r,q)_{0,\infty},z)&\subset\coinn((r,q)_{0,\infty}) \mbox{, for every $z\in Z$, $n\in\N$.}
\end{align}

We say that $S\subset\coinn((r,q)_{0,\infty},z)$ is \emph{$(p,\ep,n,z)$-separated} for $(r,q)_{0,\infty}$, if it is a $(p,\ep,n)$-separated subset (defined in Section~\ref{s:4}). We say that $R\subset \coinn((r,q)_{0,\infty},z)$ is \emph{$(p,\ep,n,z)$-spanning} for $(r,q)_{0,\infty}$, if it is a $(p,n,\ep)$-spanning subset (defined in Section~\ref{s:4}). For the definition of pseudometric $p_n$, see \eqref{eq:15}.

\begin{definition}\label{d:5.4}
Let $X=(X,\U)$, $Z=(Z,\V)$ be compact uniform spaces and $(r,q)_{0,\infty}$ be a sequence of selected pairs $q_j,r_j:Z\to X$, $j\in\N\cup\{0\}$, where $r_j$, $j\in\N\cup\{0\}$, are single-valued surjections and $q_j$, $j\in\N\cup\{0\}$, are single-valued maps.

Denoting, for a given $z\in Z$, by $s_Z((r,q)_{0,\infty},p,\ep,n,z)$ the largest cardinality of a $(p,\ep,n,z)$-separated set for $(r,q)_{0,\infty}$, resp. by $r_Z((r,q)_{0,\infty},p,\ep,n,z)$ the smallest cardinality of a $(p,\ep,n,z)$-spanning set for $(r,q)_{0,\infty}$, we take (cf. \eqref{eq:19})
\begin{align*}
\hp^Z((r,q)_{0,\infty},p,\ep,z)&:= \limsup_{n\to\infty} \frac1{n} \log s_Z(\vp_{0,\infty},p,\ep,n,z)\\
&=\limsup_{n\to\infty} \frac1{n}\log r_Z(\vp_{0,\infty},p,\ep,n,z).
\end{align*}
The \emph{parametric pointwise entropy} $\hp^Z((r,q)_{0,\infty})$ of $\vp_{0,\infty}=(r,q)_{0,\infty}$ is defined to be 
\begin{equation*}
\hp^Z((r,q)_{0,\infty}):= \sup_{z\in Z} \sup_{p\in \V, \ep>0} \hp^Z((r,q)_{0,\infty},p,\ep,z).
\end{equation*}
The \emph{parametric pointwise entropy} $\hm^Z((r,q)_{0,\infty})$ of $\vp_{0,\infty}=(r,q)_{0,\infty}$ is defined to be 
\begin{align*}
\hm^Z((r,q)_{0,\infty},p,\ep,z)&:= \sup_{p\in \V, \ep>0} \limsup_{n\to\infty} \frac1{n} \log [\sup_{z\in Z} s_Z(\vp_{0,\infty},p,\ep,n,z)]\\
&=\sup_{p\in \V, \ep>0}\limsup_{n\to\infty} \frac1{n}\log [\sup_{z\in Z} r_Z(\vp_{0,\infty},p,\ep,n,z)].
\end{align*}
\end{definition}

Defining the pseudometric $p_n^b$ on $Z$ as
\begin{align*}
p_n^b(u,v) &:= \pnH\left(\coinn((r,q)_{0,\infty},u),\coinn((r,q)_{0,\infty},v)\right) \\
&=\max\{\sup_{\overline{u}\in\coinn((r,q)_{0,\infty},u)} \inf_{\overline{v}\in\coinn((r,q)_{0,\infty},v)}  p_n(\overline{u},\overline{v}),\\
&\sup_{\overline{v}\in\coinn((r,q)_{0,\infty},v)} \inf_{\overline{u}\in\coinn((r,q)_{0,\infty},u)}  p_n(\overline{v},\overline{u})\},
\end{align*}
where $p_n(\overline{u},\overline{v}):= \max \{p(u_i,v_i):i=1,\ldots,n-1\}$, $\overline{u}=(u_0,\ldots,u_{n-1})\in Z^n$, $\overline{v}=(v_0,\ldots,v_{n-1})\in Z^n$, $u=u_0\in Z$, $v=v_0\in Z$, we can formulate a ``branch'' version of Definition~\ref{d:5.3} as follows.

\begin{definition}\label{d:5.5}
Let $X=(X,\U)$, $Z=(Z,\V)$ be compact uniform spaces and $(r,q)_{0,\infty}$ be a sequence of selected pairs $q_j,r_j:Z\to X$, $j\in\N\cup\{0\}$, where $r_j$, $j\in\N\cup\{0\}$, are single-valued surjections and $q_j$, $j\in\N\cup\{0\}$, are single-valued maps.

Denoting by $s((r,q)_{0,\infty},p^b_n,\ep,Z)$ the supremum of  cardinalities of $(p_n^b,\ep)$-separated sets of $Z$, resp. by $r((r,q)_{0,\infty},p^b_n,\ep,Z)$ the smallest cardinality of a $(p^b_n,\ep)$-spanning set in $Z$, the \emph{parametric branch entropy} $\hi^Z((r,q)_{0,\infty})$ of $(r,q)_{0,\infty}$ is defined to be (the equality follows from Lemma~\ref{l:sep-span}, cf. Definition~\ref{d:5.3})
\begin{align*}
\hi^{Z}((r,q)_{0,\infty})&:= \sup_{p\in\V, \ep>0} \limsup_{n\to\infty} \frac1{n} \log \fin(s((r,q)_{0,\infty},p^b_n,\ep,Z))\\
&=\sup_{p\in\V,\ep>0} \limsup_{n\to\infty} \frac1{n} \log \fin(r((r,q)_{0,\infty},p^b_n,\ep,Z)),
\end{align*}
where
\begin{equation*}
\fin(\kappa)=
\begin{cases}
&\kappa, \mbox{ if $\kappa$ is a finite cardinal number,}\\
&\infty, \mbox{ if $\kappa$ is an infinite cardinal number.}
\end{cases}
\end{equation*}
\end{definition}

\begin{remark}
If $q_j,r_j:Z\to X$, $j\in\N\cup\{0\}$, are still (single-valued) continuous maps, then by the similar arguments (when applying Lemma~\ref{l:1b}, instead of Lemma~\ref{l:1}) as in Remark~\ref{r:5.4}, the operator $\fin$ can be omitted and we can use for the cardinalities the attribute maximal rather than supremal in Definition~\ref{d:5.5}.
\end{remark}

\begin{remark}
Observe that if $r_j^{-1}$, $j\in\N\cup\{0\}$, are single-valud maps, then $r_j$ must be homeomorphisms, by which $\hp^Z((r,q)_{0,\infty})=\hm^Z((r,q)_{0,\infty})=0$, and $h^Z((r,q)_{0,\infty})= \hi^Z((r,q)_{0,\infty})$. Otherwise, if $q_j\circ r^{-1}_j$, $j\in\N\cup\{0\}$, are single-valued, but $r_j$ are not homeomorphisms, then $\hp^Z((r,q)_{0,\infty})$, $\hm^Z((r,q)_{0,\infty})$ do not necessarily vanish and $\hi^Z((r,q)_{0,\infty})$ does not necessarily equal to $h_Z((r,q)_{0,\infty})$ (see Example~\ref{e:4.3} above).
\end{remark}

\begin{figure}[h]
  \centering
  \includegraphics[width=0.8\textwidth]{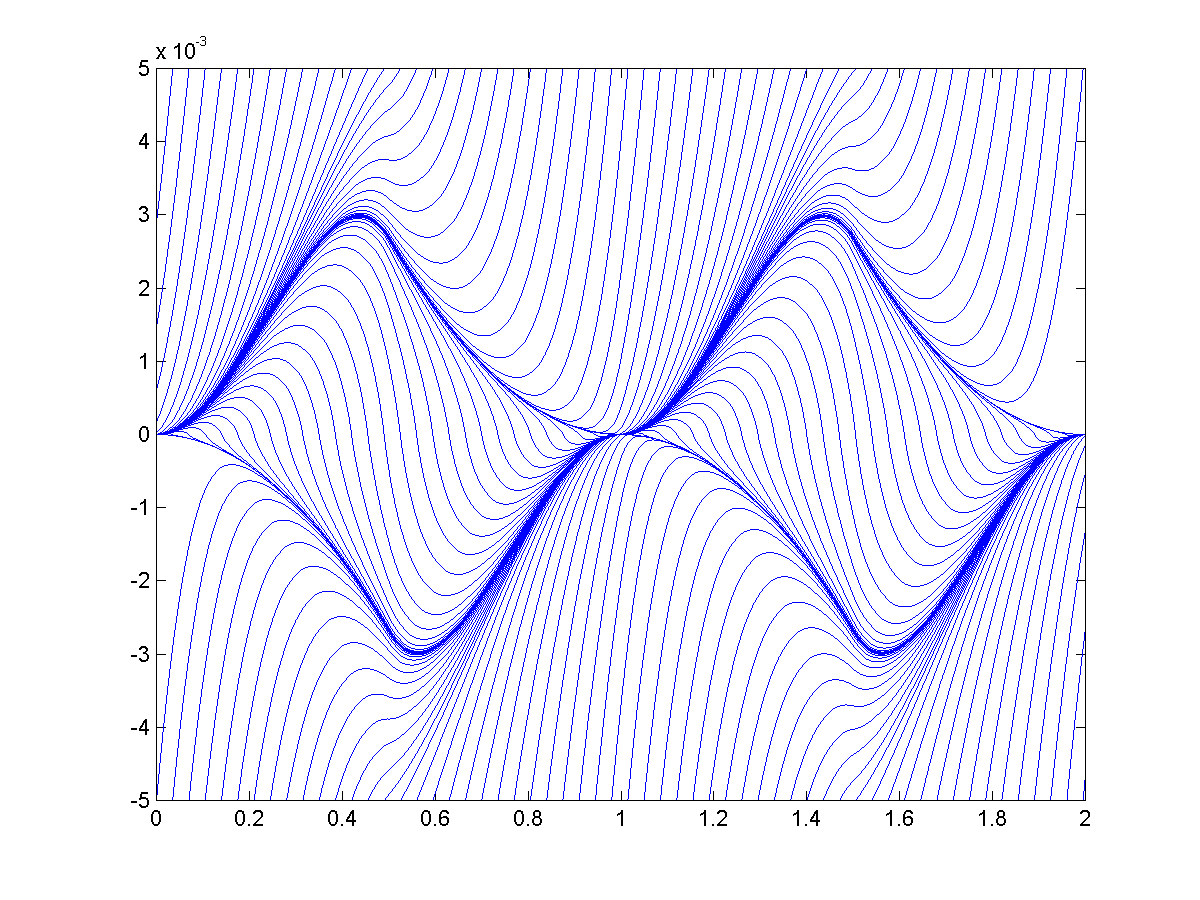}
  \caption{Part of trajectories of the equation $x' = \sqrt{\abs{x}}-\frac1{8\pi}\abs{\arcsin(\sin \pi t)}$.}
  \label{f:1}
\end{figure}

\begin{theorem}\label{t:5.2}
Let $X$, $Z$ be compact Hausdorff spaces and $(r,q)_{0,\infty}$ be a sequence of selected pairs $q_j,r_j:Z\to X$, $j\in\N\cup\{0\}$, where $r_j$, $j\in\N\cup\{0\}$, are single-valued surjections and $q_j$, $j\in\N\cup\{0\}$, are single-valued maps. Then the inequalitites
\begin{equation*}
\hp^Z((r,q)_{0,\infty})\leq \hm^Z((r,q)_{0,\infty}) \leq h_Z((r,q)_{0,\infty}) \leq \hm^Z((r,q)_{0,\infty}) + \hi^Z((r,q)_{0,\infty})
\end{equation*}
are satisfied for the topological entropies in the sense of Definitions~\ref{d:4.2}, \ref{d:5.4}, \ref{d:5.5}, respectively.
\end{theorem}

\begin{proof}
The inequalities $\hp^Z((r,q)_{0,\infty})\leq \hm^Z((r,q)_{0,\infty}) \leq h_Z((r,q)_{0,\infty})$ follow directly from the definitions of the respective entropies.

The inequality $h_Z((r,q)_{0,\infty}) \leq \hm^Z((r,q)_{0,\infty}) + \hi^Z((r,q)_{0,\infty})$ can be verified quite analogously as in the proof of Theorem~\ref{t:5.1} for \eqref{eq:18b}.

\end{proof}

\begin{remark}
One can easily check that
\begin{equation*}
\hp^Z((r,q)_{0,\infty}) = \hm^Z((r,q)_{0,\infty}) = h_Z((r,q)_{0,\infty}) = \log 2
\end{equation*}
holds for the sequence $(r,q)_{0,\infty}$ considered in Example~\ref{e:4.3}.

The diagrams of such and more sophisticated sequences of maps $q_j$, $r_j$, $j\in\N\cup\{0\}$, correspond to the graphs of solutions of (possibly impulsive) differential equations and inclusions (see e.g. \cite{A2,A3,AJ,AL2,JS,RS}). Positive values of entropies thus can refer about the complexity of their trajectories behaviour.

In Figure~\ref{f:1}, a part of trajectories of the differential equation $x'=\sqrt{\abs{x}} - \frac1{8\pi}\abs{\arcsin(\sin \pi t))})$ is plotted such that $h((r,q)_{0,\infty})=0$, but $h_Z((r,q)_{0,\infty})=\infty$.
\end{remark}

\section{Concluding remarks}
As already emphasized, Definitions~\ref{d:3.8} and \ref{d:5.1}-\ref{d:5.3} of parametric entropy deal with arbitrary multivalued maps in compact Hausdorff spaces. As far as we know, there is only one another definition, but of a different type, which concerns arbitrary multivalued autonomous maps in metric spaces (see \cite{CMM}).

It is not hard to prove the inequalities
\begin{equation*}
h^{\operatorname{span}}_{\operatorname{CMM}}(\vp) \leq h^{\operatorname{sep}}_{\operatorname{CMM}}(\vp) \leq h(\vp),
\end{equation*}
for arbitrary multivalued maps $\vp:X\multimap X$ in a compact metric space $X$, where $h^{\operatorname{span}}_{\operatorname{CMM}}$, $h^{\operatorname{sep}}_{\operatorname{CMM}}$ stand for the topological entropies defined in \cite{CMM}, and $h$ satisfies a special (nonparametric) version of Definition~\ref{d:3.8}.

Since analogous inequalities can be also obtained for some further definitions of even parametric topological entropy (see e.g. \cite{AL2}), all the related illustrative examples, including those in \cite{CMM}, can serve as lower estimates, when applying Definition~\ref{d:3.8}.

Definition~\ref{d:3.8} certainly generalizes the definitions of topological entropy in \cite{Ci,JS,KLP,NS,Pw,PLB,Ma,MZ} for arbitrary (i.e. also discontinuous) single-valued maps having a weaker regularity than continuity, like almost or piecewise continuity. On the other hand, we have shown in Section~\ref{s:4} that, for effective calculations or estimates, at least a certain amount of regularity like upper semicontinuity should be imposed on given multivalued maps under consideration. Since u.s.c. maps become in the single-valued case continuous, the obtained results for calculations in the quoted papers for discontinuous single-valued maps cannot be deduced from ours. Definitions~\ref{d:5.1} and~\ref{d:5.2} have no meaning for single-valued maps at all.

The Markov (single-valued) interval functions belong to a suitable class for calculating the topological entropy. Recently, their definition was extended to a multivalued setting and applied e.g. in \cite{AK,EK}. Although a further extension was done for multivalued maps on compact metric spaces in \cite{BCK}, appropriate applications for them seem to be technically difficult. Topological entropy of single-valued maps can be rather curiously explored also by means of topological invariants like the Conley index, in the frame of multivalued maps (see e.g. \cite{Bak,Vi}).

Let us finally note that also the related notions to topological entropy, the metric entropy and the topological pressure, were recently examined for multivalued maps in compact metric spaces (see e.g. \cite{VS,XZ}). Their parametric versions will be examined by ourselves elsewhere.

\bibliography{AL-GenOfKT-arxiv}\bibliographystyle{siam}

\end{document}